\documentclass[reqno,11pt]{amsart}
\usepackage{amsmath, latexsym, amsfonts, amssymb, amsthm, amscd}
\usepackage{graphics,epsf,psfrag}
\usepackage{graphics,epsf,psfrag}
 \usepackage[utf8]{inputenc}
\usepackage{stmaryrd}
\usepackage{mathabx}
\usepackage{hyperref}
\usepackage{xcolor}

\renewcommand{\bar}{\overline}

\newcommand{\lint}{\llbracket}
\newcommand{\rint}{\rrbracket}

\numberwithin{equation}{section}

\newtheorem{theorema}{Theorem}

\newtheorem{theorem}{Theorem}[section]
\newtheorem{lemma}[theorem]{Lemma}
\newtheorem{proposition}[theorem]{Proposition}
\newtheorem{cor}[theorem]{Corollary}
\newtheorem{rem}[theorem]{Remark}

\newcommand{\dd}{\mathrm{d}}
\newcommand{\ind}{\mathbf{1}}

\renewcommand{\tilde}{\widetilde}
\renewcommand{\hat}{\widehat}

\newcommand{\cA}{{\ensuremath{\mathcal A}} }
\newcommand{\cB}{{\ensuremath{\mathcal B}} }
\newcommand{\cF}{{\ensuremath{\mathcal F}} }

\newcommand{\cD}{{\ensuremath{\mathcal D}} }

\newcommand{\cZ}{{\ensuremath{\mathcal Z}} }

\newcommand{\bP}{{\ensuremath{\mathbf P}} }

\DeclareMathSymbol{\leqslant}{\mathalpha}{AMSa}{"36} 
\DeclareMathSymbol{\geqslant}{\mathalpha}{AMSa}{"3E} 
\DeclareMathSymbol{\eset}{\mathalpha}{AMSb}{"3F}     

\newcommand{\bbE}{{\ensuremath{\mathbb E}} }

\newcommand{\bbN}{{\ensuremath{\mathbb N}} }

\newcommand{\bbP}{{\ensuremath{\mathbb P}} }

\newcommand{\bbR}{{\ensuremath{\mathbb R}} }

\newcommand{\bbZ}{{\ensuremath{\mathbb Z}} }

\newcommand{\gd}{\delta}

\newcommand{\gep}{\varepsilon}       

\newcommand{\gl}{\lambda}

\makeatletter
\def\captionfont@{\footnotesize}
\def\captionheadfont@{\scshape}

\long\def\@makecaption#1#2{%
  \vspace{2mm}
  \setbox\@tempboxa\vbox{\color@setgroup
    \advance\hsize-6pc\noindent
    \captionfont@\captionheadfont@#1\@xp\@ifnotempty\@xp
        {\@cdr#2\@nil}{.\captionfont@\upshape\enspace#2}%
    \unskip\kern-6pc\par
    \global\setbox\@ne\lastbox\color@endgroup}%
  \ifhbox\@ne 
    \setbox\@ne\hbox{\unhbox\@ne\unskip\unskip\unpenalty\unkern}%
  \fi
  \ifdim\wd\@tempboxa=\z@ 
    \setbox\@ne\hbox to\columnwidth{\hss\kern-6pc\box\@ne\hss}%
  \else 
    \setbox\@ne\vbox{\unvbox\@tempboxa\parskip\z@skip
        \noindent\unhbox\@ne\advance\hsize-6pc\par}%
\fi
  \ifnum\@tempcnta<64 
    \addvspace\abovecaptionskip
    \moveright 3pc\box\@ne
  \else 
    \moveright 3pc\box\@ne
    \nobreak
    \vskip\belowcaptionskip
  \fi
\relax
}
\makeatother
\def\writefig#1 #2 #3 {\rlap{\kern #1 truecm
\raise #2 truecm \hbox{#3}}}

\title{Critical Gaussian Multiplicative Chaos revisited}

\author{Hubert Lacoin}
\address{
  IMPA, Institudo de Matem\'atica Pura e Aplicada, Estrada Dona Castorina 110
Rio de Janeiro, CEP-22460-320, Brasil. 
}

\begin{document}

 \begin{abstract}
 We present new, short and self-contained proofs of the convergence  (with an adequate renormalization) of four different sequences to the critical Gaussian Multiplicative Chaos:(a) the derivative martingale (b) the critical martingale (c) the exponential of the mollified field (d) the subcritical Gaussian Multiplicative Chaos.  \\[4pt]
  2010 \textit{Mathematics Subject Classification: 60F99,  	60G15,  	82B99.}\\
\textit{Keywords: Random distributions, $\log$-correlated fields, Gaussian Multiplicative Chaos.} 
 \end{abstract}

\maketitle

 
 \section{Introduction}

We consider $K: \bbR^d\times \bbR^d \to (-\infty,\infty]$ to be a positive definite kernel on $\bbR^d$ ($d\ge 1$ is fixed) which admits a decomposition in the following form
 \begin{equation}\label{fourme}
  K(x,y):= \log \frac{1}{|x-y|}+L(x,y),
 \end{equation}
where $L$ is a continuous function ($\log 0=\infty$ by convention). 
A kernel $K$ is  positive definite if for any  $\rho\in C_c(\bbR^d)$ ($\rho$ continuous with compact support)
 $\int_{\cD^2 }  K(x,y) \rho(x)\rho(y)\dd x \dd y\ge 0.$
Given a centered Gaussian field $X$ with covariance $K$, the \textit{critical Gaussian Multiplicative Chaos} (or critical GMC) is heuristically defined as 
 \begin{equation}\label{defGMC}
 e^{\alpha X(x)} \dd x  \quad   \text{ with } \quad    \alpha=\sqrt{2d}.
 \end{equation}
The main difficulty that comes up when  trying to give an interpretation to the expression in  \eqref{defGMC} is 
that a field with a covariance of the type \eqref{fourme} can be defined only as a random distribution and thus $X$ is not defined pointwise.
The problem of rigorously defining the  GMC was first considered by Kahane in \cite{zbMATH03960673}. The standard procedure (for any $\alpha\ge 0$) is to  use a sequence of  approximation of $X$ 
and then pass to the limit. Mostly  two kinds of approximation of $X$ have been considered in the literature: 
\begin{itemize}
 \item [(A)] A mollification of the field, $X_{\gep}$,  via convolution with a smooth kernel on scale $\gep$,
 \item [(B)] A martingale approximation, $X_t$, via an integral decomposition of the kernel $K$.
\end{itemize}
Rigorous definitions of $X_{\gep}$ and $X_t$ are given in the next subsections. We define  $M^{\alpha}_\gep$ by
\begin{equation}\label{defalpha}
M^{\alpha}_\gep(f):=\int_{\bbR^d} f(x) e^{\alpha X_{\gep}(x)-\frac{\alpha^{2}}{2}\bbE[X_{\gep}(x)]} \dd x
\end{equation}
and $M^{\alpha}_t$  in the same manner using the martingale approximation $X_t$. Then we  let either $\gep\to 0$ or $t\to \infty$.
When $\alpha\in [0,\sqrt{2d})$, it has been proved \cite{Natele,Shamov} that both $M^{\alpha}_t$ and $M^{\alpha}_\gep$  converge to a nontrivial limiting measure $M^{\alpha}$, which does not
depend on the 
mollifier. We refer the reader to the introduction of \cite{Natele} or to the review \cite{RVreview} for an account of the steps leading to this result. We also refer to  \cite[Section 5]{RVreview} for motivations to study GMC coming from various fields in including theoretical physics.

\medskip

When $\alpha=\sqrt{2d}$, the limiting behavior is different. We have $\lim_{\gep\to 0}M^{\sqrt{2d}}_\gep(f)=0$ for any bounded measurable $f$ and a 
renormalization procedure is required to obtain a nontrivial limit. 
More precisely it has been proved that $\sqrt{\pi\log (1/\gep)/2} M^{2d}_\gep$ converges to a nontrivial limit that does not depend on the mollifier. This was achieved in three steps, starting 
with the martingale approximation $X_t$: 
\begin{itemize}
 \item [(i)] In \cite{MR3262492}, the convergence of the random (signed) measure  
 \begin{equation}\label{derivtiv}
 D_t(\dd x):=(\sqrt{2d}\bbE[X_{t}(x)]- X_t) e^{\sqrt{2d} X_{t}(x)- d \bbE[X_{t}(x)]} \dd x,
 \end{equation}
 to a nontrivial nonnegative measure $D_\infty$ was established.
 \item [(ii)] In \cite{MR3215583}, it was proved that $\sqrt{\pi t/2}  M^{\sqrt{2d}}_t$ converges to the same limit.
 \item [(iii)] The  convergence of $\sqrt{\pi \log (1/\gep)/2} M^{\sqrt{2d}}_\gep$ towards $D_{\infty}$ was shown in  \cite{MR3613704}. 
\end{itemize}
Additionally two other convergence results have been proved:
\begin{itemize}
 \item  [(iv)] In \cite{MR3933042, MR4396197}, the convergence of $(\sqrt{2d}-\alpha)M^{\alpha}$ towards $2D_{\infty}$ when $\alpha \uparrow \sqrt{2d}$.
 \item [(v)] In  \cite{powail}, the convergence of $D_{\gep}$ (defined as the mollified analogue of
 $D_t$) to $D_{\infty}$. 
\end{itemize}
We refer to \cite{MR4396197} for a thorough review of these results.  
 The objective of this paper is to present new simple and short  proofs of  statements $(i)$ to $(iv)$ above with a presentation as self-contained as possible.
We believe that the technique we present can also be adapted to  provide an alternative proof of $(v)$.
 Along the way we present proof for an asymptotic upper bound for $\log$-correlated fields (Proposition \ref{tiptop}) which is much shorter that what has appeared so far in the literature.
 
 \medskip
 
 While this manuscript is focused on the Gaussian case, let us mention that analogous random measures have been obtained considering the exponential of nongaussian $\log$-corelated fields \cite{jego1, jego2, junnila2018m}. In particular \cite{jego2} establishes results which are very similar to those mentionned above for the exponentiation of the square root of the local time of a planar Brownian Motion at the critical parameter.

\subsection{Mollification of the field and limit}\label{molly}

\subsubsection*{Log-correlated fields  defined as distributions}

Since $K$ is infinite on the diagonal, it is not possible to define  a Gaussian field indexed by $\bbR^d$ with covariance function $K$.
We consider instead a process indexed by test functions. 
We define $\hat K$, a bilinear form on  $C_c(\bbR^d)$, by   
\begin{equation}\label{hatK}
 \hat K(\rho,\rho')=
\int_{\cD^2}  K (x,y)\rho(x)\rho'(y)\dd x \dd y.
\end{equation}
Since $\hat K$ is positive definite, we can  consider $X= \langle X, \rho \rangle_{\rho\in C_c(\bbR^d)}$ a centered Gaussian field indexed by $C_c(\bbR^d)$ with covariance kernel given by $\hat K$.

\subsubsection*{Convolution approximation}

The random distribution $X$ can be approximated by a sequence of  functional fields - fields indexed by $\bbR^d$ - 
by the mean of convolution with a smooth kernel.
 Consider $\theta$ a nonnegative function in $C_c^{\infty}(\bbR^d)$  whose compact support is included in  $B(0,1)$ (for the remainder 
 of the paper $B(x,r)$ denotes the closed 
 Euclidean ball of center $x$ and radius $r$) 
 and which satisfies  $\int_{B(0,1)} \theta(x) \dd x=1.$
We define for $\gep>0$,
$\theta_{\gep}:=\gep^{-d} \theta( \gep^{-1}\cdot)$, set
and consider $(X_{\gep}(x))_{x\in \bbR^d}$, the mollified version of $X$, that is
\begin{equation}
 X_{\gep}(x):= \langle X, \theta_{\gep}(x-\cdot) \rangle
\end{equation}
From \eqref{hatK}, the field $X_{\gep}(\cdot)$ has covariance 
\begin{equation}\label{labig}
K_{\gep}(x,y):=\bbE[X_{\gep}(x)X_{\gep}(y)]=
\int_{\bbR^{2d}} \theta_{\gep}(x-z_1) \theta_{\gep}(y-z_2)
K(z_1,z_2)\dd z_1 \dd z_2.
\end{equation}
We set $K_{\gep}(x):=K_{\gep}(x,x)$ and use a similar convention for other covariance functions.
Since $K_{\gep}$ is infinitely differentiable, by Kolmogorov's Continuity Theorem (e.g.\  \cite[Theorem 2.9]{legallSto})  there exists a continuous modification of $X_{\gep}(\cdot)$. In the remainder of the paper, we always consider the continuous modification of a process when it exists, this ensures that integrals such as the one appearing in \eqref{forboundedfunctions} are well defined.
 We let $\mathcal B_b$ denote the bounded Borel subsets of $\bbR^d$ and $B_b$ denote the bounded Borel functions with bounded support
 \begin{equation}\label{measurable}
 B_b=B_b(\bbR^d)= \left\{ f \  \text{measurable} \ : \sup_{x\in \bbR^d} |f(x)|<\infty , \ \{ x \ : \ |f(x)|\ne 0\} \in \cB_b  \right\}.
 \end{equation}
We  define the measure     $M^{\alpha}_{\gep}$ by setting for $f\in B_b$ (recall \eqref{defalpha})
 \begin{equation}\label{forboundedfunctions}
    M^{\alpha}_{\gep}(f):=\int_{\bbR^d} f(x) e^{\alpha X_{\gep}(x)- \frac{\alpha^2}{2} K_\gep(x)}\dd x.
 \end{equation}
 We set $M^{\alpha}_{\gep}(E):=M^{\alpha}_{\gep}(\ind_{E})$ for $E\in \mathcal B_b$ and keep a similar convention for other measures.
When considering the the value $\alpha=\sqrt{2d}$ (which is most of the time) we drop the dependence in $\alpha$ in the notation, we set  $M_{\gep}= M^{\sqrt{2d}}_{\gep}$.
 Before stating the main results, we need to mention an important assumption that need to make on $K$.

\subsection{Star-scale invariance and our assumption on $K$}

We assume, in most of the paper, that the kernel $K$ has an \textit{almost star-scale invariant} part. This assumption might seem at first quite restrictive, but it has been shown in 
\cite{junnila2019} that it is \textit{locally} satisfied as soon as  $L$ is sufficiently regular. We explain in Appendix \ref{sobolev} how this allows to extend the results 
 to log-correlated field on arbitry domains, provided that $L$ satisfies the required regularity assumption.
Following a terminology introduced in \cite{junnila2019}, we say that a the kernel $K$ defined on $\bbR^d$ is \textit{almost star-scale invariant}
if it can be written in the form  
\begin{equation}\label{iladeuxstar}
 \forall x,y \in \bbR^d,\  K(x,y)=\int^{\infty}_{0} (1-\eta_1 e^{-\eta_2 t}) \kappa(e^{t}(x-y))\dd t,
\end{equation}
 where  $\eta_1\in[0,1]$ and $\eta_2>0$ are constants and the function $\kappa\in C^{\infty}_c(\bbR^d)$ is radial, nonnegative and definite positive. More precisely we assume the following:
\begin{itemize}
 \item [(i)] $\kappa\in C_c^{\infty}(\bbR^d)$ and there exists $\tilde \kappa : \ \bbR^+\to [0,\infty)$ such that $\kappa(x):=\tilde \kappa(|x|)$,
 \item [(ii)] $\tilde \kappa(0)=1$ and $\tilde\kappa(r)=0$ for $r\ge 1$,
 \item [(iii)]The mapping $(x,y)\mapsto \kappa(x-y)$ defines a positive definite-kernel on $\bbR^d\times \bbR^d$.
\end{itemize}
We say furthermore that a kernel $K$   \textit{has an almost star-scale invariant part}, if 
\begin{equation}\label{3star}
 \forall x,y\in \bbR^{d}, \   K(x,y)=K_0(x,y)+ \bar K(x,y)
\end{equation}
where $\bar K(x,y)$ is an almost star-scale invariant kernel and $K_0$ is H\"older continuous on $\bbR^{2d}$ and positive  definite. 
 Given $K$ with an almost star-scale invariant part, and using the decomposition \eqref{iladeuxstar} for $\bar K$, we set
\begin{equation}\label{defqt}
 Q_t(x,y):=  \kappa(e^{t'}(x-y))
 \end{equation}
where $t'$ is defined as the unique positive solution of $ t'-\frac{\eta_1}{\eta_2}(1-e^{-\eta_2 t'})=t.$
We set 
\begin{equation}\label{kkttt}\begin{split}
 K_t(x,y)&:=K_0(x,y)+ \int^t_0 Q_s(x,y) \dd s \\
 &=K_0(x,y) + \int_{0}^{t'} (1-\eta_1 e^{-\eta_2 s}) \kappa(e^{s}(x-y))\dd s=:K_0(x,y)+ \bar K_t(x,y).
\end{split}\end{equation}
Note that we have $\lim_{t\to \infty} K_t(x,y)=K(x,y)$ and $\bar K_t(x)=t$.
If $K$ satisfies \eqref{iladeuxstar} then
\begin{equation}
 L(x,y):= K(x,y)+\log |x-y|,
\end{equation}
can be extended to a continuous function on $\bbR^{2d}$, so that a kernel $K$ with an almost star-scale invariant part
can always be written in the form \eqref{fourme}.

\begin{rem}
 There is an obvious conflict of notation between $K_t$ introduced above  and $K_{\gep}$ introduced in \eqref{labig} 
 and the same can be said about $X_t$ and $M_t$ introduced in the next section. However this abuse should not cause any confusion since we keep using the letter 
 $\gep$ for quantities related to the mollified field $X_{\gep}$ and latin letters for quantities related to the martingale approximation $X_t$.  
\end{rem}

\subsection{Convergence of the convolution approximation}

The main result whose proof we wish to present in this paper is the convergence of the measure $M_{\gep}$ - properly rescaled - 
towards a limit $M'$. The principal task is to show that $M_{\gep}(E)$ converges for any fixed $E$.

\begin{theorem}\label{mainres}
 If $X$ is a centered Gaussian field whose covariance kernel $K$  has an almost star-scale invariant part and $E\in \mathcal B_b$ is  
 of positive Lebesgue measure then
there exists an a.s.\ positive random variable $M'(E)$ such that the following convergence holds in probability
 \begin{equation}\label{inproba}
  \lim_{\gep \to 0}\sqrt{\frac{\pi\log (1/\gep)}{2}}M_{\gep}(E)=M'(E).
 \end{equation}
 The limit does not depend on the specific mollifier $\theta$ used to define $X_{\gep}$.
\end{theorem}

\noindent In order to extend the statement to a  convergence of measure, we need to specify a topology. We say that a sequence of locally finite measures 
$(\mu_n)$ converges weakly to $\mu$ if
\begin{equation}\label{wikiwik}
\forall f \in C_c(\bbR^d), \quad  \lim_{n\to \infty} \int_{\bbR^d} f(x)\mu_n(\dd x)= \int_{\bbR^d} f(x) \dd \mu(\dd x).
\end{equation}
The topology of weak convergence for locally finite nonnegative measures is metrizable and separable, hence we can associate to it a notion of convergence in probability for a sequence of 
random measure. An example of metric generating the topology is given in Appendix \ref{weakconvmeas}.
For $E\in \mathcal B_b$ we let $|E|:= \max_{x,y\in E}|x-y|$ denote the diameter of $E$. We introduce the gauge function $\phi$ defined on $[0,\infty)$ by
\begin{equation}\label{phifou}
 \phi(u)=\begin{cases} \log \log (1/u)[\log (1/u)]^{-1/4} \quad &\text{ if } u\le e^{-e},\\ 
e^{-1/4} \quad &\text{ if } u\ge e^{-e}.
         \end{cases}
\end{equation}

\begin{theorem}\label{mainresprim}
There exists a locally finite random measure $M'$ with dense support such that 
 \begin{equation*}
    \lim_{\gep \to 0}\sqrt{\frac{\pi\log (1/\gep)}{2}}M_{\gep}=M',
 \end{equation*}
weakly in probability. Furthermore $M'$ has no atoms and we have any fixed $R>0$
\begin{equation}\label{gageli}
\sup_{E\subset B(0,R)} M'(E)/\phi(|E|)<\infty \quad \bbP-a.s.
\end{equation}

\end{theorem}

\begin{rem}
 Most of our efforts are dedicated to the proof of Theorem \ref{mainres} and of
\eqref{gageli}. The (standard) task of showing that the convergence of $M_{\gep}(E)$ for every $E$ implies  weak convergence of measures
is  performed in Appendix \ref{weakconvmeas} for completeness. The fact that $M'$ has dense support is a consequence of the a.s.\ positivity of $M'(E)$ in Theorem \ref{mainres}.
\end{rem}

\begin{rem}
Theorem \ref{sobolevteo} extends Theorem \ref{mainresprim} to all $\log$-correlated fields defined on an open domain $\mathcal D$ provided that the function $L$ in 
\eqref{fourme} has sufficient regularity. This condition on $L$ is satisfied for most if not all known applications of
Gaussian Multiplicative Chaos \cite[Section 5]{RVreview},  and in particular by planar free fields considered in \cite{MR3933042, lioucrit, HRV}.  Without any modification to the proof, it is also possible to extend our result to chaos defined with a reference measure which is not Lebesgue measure, provided that it is ``sufficiently spread out'', see the discussion in Section \ref{deesk}.
\end{rem}

\subsection{The martingale decomposition of $X$} \label{martindecoco}
Given $K$ a kernel on $\bbR^d$ with an almost star-scale invariant part, 
we define $(X_{t}(x))_{x\in \bbR^d, t\ge 0}$ to be a centered Gaussian field with covariance given by 
(using the  notation $a\wedge b:=\min(a,b)$, $a\vee b:=\max(a,b)$)
\begin{equation}
 \bbE[X_t(x)X_s(y)]= K_{s\wedge t}(x,y).
\end{equation}
Since  $(s,t,x,y)\mapsto K_{s\wedge t}(x,y)$ is H\"older continuous, 
the field admits a continuous modification.
We let  
$\mathcal F_t:= \sigma\left( (X_s(x))_{ x\in \bbR^d ,s\in [0,t]}\right)$
denote the natural filtration associated with $X_{\cdot}(\cdot)$.
The process $X$ indexed by $C_c(\bbR^d)$ and defined by
$\langle X,\rho\rangle= \lim_{t\to \infty} \int_{\bbR^d} X_t(x) \rho(x)\dd x,$
is a centered Gaussian field, so that $X_t$ is an approximation sequence for a $\log$-correlated field with covariance $K$.
We define the measure     $M^{\alpha}_{t}$ by setting for $f\in B_b$ (recall \eqref{measurable})  
 \begin{equation}\label{defalfat}
    M^{\alpha}_{t}(f):=\int_{\bbR^d} f(x)  e^{\alpha X_{t}(x)- \frac{\alpha^2}{2} K_t(x)} \dd x,
 \end{equation}
 and set $M_t= M^{\sqrt{2d}}_t$.
By independence of the increments of $X$ (see Section \ref{considerations}), $M_t(f)$ is an $(\cF_t)$-martingale. 
We also consider the  \textit{derivative martingale} $D_t$, defined by 
\begin{equation}\label{laderivvv}
 D_t(f)= \int_{\bbR^d}f(x)\left( \sqrt{2d}K_t(x)-X_t \right)  e^{\sqrt{2d} X_{t}(x)- d K_t(x)} \dd x.
\end{equation}
We have $D_t(f):=- \partial_{\alpha}     M^{\alpha}_{t}(f)|_{\alpha=\sqrt{2d}}$ (hence the name). 
To prove Theorem \ref{mainres} we first  establish the convergence of $D_t$. 
We let $\gl$ denote the Lebesgue measure and define 
\begin{equation}
 \bar D_\infty(f):=\limsup_{t\to \infty} D_{t}(f).
\end{equation}
The following result, that is, the convergence of $D_t$ to a non-trivial limit, is also  the first step to prove Theorem \ref{mainres}.
 \begin{theorem}\label{derivmartin}
For any fixed $E\in \cB_b$ we have almost surely
 \begin{equation}
\bar  D_\infty(E)= 
 \lim_{t\to \infty} D_{t}(E)\in[0,\infty),
 \end{equation}
 If $\gl(E)>0$, then $\bbP[ \bar D_{\infty}(E)>0=1]$. Recalling   \eqref{phifou}, for every $R>0$ we have
\begin{equation}\label{lagaage}
 \sup_{E\subset B(0,R)} \bar D_{\infty}(E)/\phi(|E|)<\infty.
 \end{equation}
\end{theorem}
\noindent  Applying Proposition \ref{weako}  to the positive part $D_t$ (it coincides with $D_t$ for  large $t$ by Proposition \ref{tiptop}) we can deduce from the above the weak convergence of the measure $D_t$.

 \begin{theorem}\label{measureconv}
 There exists a measure $D_{\infty}$ such that almost surely, $D_t$ converges weakly to $D_{\infty}$. Furthermore we have
for every $f\in  B_b$, $\bbP[ \bar D_\infty(f)=   D_\infty(f)]=1.$
The measure $D_{\infty}$ has dense support and satisfies \eqref{lagaage} (in particular it has no atoms).
 \end{theorem}
\noindent We also prove that $M_t$, when appropriately rescaled, converges to the same limit.

 \begin{theorem}\label{critmartin}
For any fixed $E\in \mathcal B_b$ we have the following convergence in probability, 
 \begin{equation}
  \lim_{t\to \infty} \sqrt{ \frac{\pi t}{2}}M_{t}(E)=
\bar D_{\infty}(E).
 \end{equation}
 Furthermore  $\sqrt{ \frac{\pi t}{2}}M_{t}$ converges weakly in probability to $D_{\infty}$.

\end{theorem}

\subsection{Convergence of subcritical chaos when $\alpha$ tends to $\sqrt{2d}$}

The last result we wish to present is the convergence of subcritical chaos, when properly renormalized, towards
the same limit $D_{\infty}$ when $\alpha \uparrow \sqrt{2d}$.
Let us recall the definition of the subcritical chaos $M^{\alpha}$.

\begin{theorema}[{\cite[Theorem 1]{Natele}}]
\label{rappel}
Let $X$ be a $\log$-correlated field (with covariance $K$ of the form \eqref{fourme})
 and  $M^{\alpha}_\gep$ be defined by \eqref{forboundedfunctions}. There exists a random measure $M^{\alpha}$ such that for every $E\in \mathcal B_b$ the following convergence holds in $L^1$ (for any choice of mollifier)
 \begin{equation}
 \lim_{\gep \to 0} M^{\alpha}_\gep(E)=M^{\alpha}(E).
 \end{equation}
 If $K$ is an almost-star scale invariant kernel, and $X_t$ is the martingale approximation of $X$ then  for every $E\in \mathcal B_b$, the following convergence holds in $L^1$ and almost surely
 \begin{equation}
   \lim_{t\to \infty} M^{\alpha}_t(E)=M^{\alpha}(E).
 \end{equation}

\end{theorema}

\noindent We can now formulate our result. In the statement below  $M^{\alpha}$ and $M'$ are the limiting measures obtained in Theorems \ref{rappel} and \ref{mainresprim} respectively. The notation $\alpha \uparrow \sqrt{2d}$  means that the limit is taken with $\alpha<\sqrt{2d}$.

\begin{theorem}\label{ziklob}
If $K$ is an almost-star scale invariant kernel and $X$ is a field with covariance $K$, then
for every $E\in \mathcal B_b$ we have the following convergence in probability
 \begin{equation}
 \lim_{\alpha  \uparrow\sqrt{2d}} \frac{M^{\alpha}(E)}{\sqrt{2d}-\alpha}=  2M'(E).
 \end{equation}
\end{theorem}

\subsection{Review of the literature and organization of the paper}

The results presented above are not new. Theorems \ref{derivmartin}-\ref{measureconv} correspond to \cite[Theorem 4]{MR3262492}
Theorem \ref{critmartin} corresponds to  \cite[Theorem 5]{MR3215583}, Theorems  \ref{mainres}-\ref{mainresprim} can be obtained by combining this last result with 
\cite[Theorem 1.1 and Theorem 4.4]{MR3613704} and Theorem \ref{ziklob} corresponds to \cite[Theorem 3.1]{MR4396197}.
The aim of this paper is to expose shorter and simpler proofs, with a presentation as self-contained and elementary 
as possible.

\medskip

In Section \ref{prelim}, we present a variety of technical results which are needed for our proofs. This includes standard probability textbook results, as well as technical estimates which are specific to our problem.
The proof of the latter are displayed in Appendix \ref{teklesproofs}.
Sections \ref{dddt}, \ref{mmmt}, \ref{mmmgep} and \ref{subkrit} respectively  display the proofs of Theorems \ref{derivmartin}, \ref{critmartin}, \ref{mainres} and \ref{ziklob}.
Appendix \ref{weakconvmeas} explains how the convergence result for the measure of a fixed Borel set implies the weak convergence of the measure. In 
Appendix \ref{sobolev}, we replicate an argument from \cite{junnila2019}, to extend Theorems \ref{mainres}-\ref{mainresprim} and \ref{ziklob} to regular $\log$-correlated 
fields on arbitrary domains.

\medskip

Our starting point to prove the convergence of $D_t$ and $M^{}_t$ is the same as in \cite{MR3262492, MR3215583}, 
we consider truncated versions  $D^{(q)}_t$ and $M^{(q)}_t$ of $M_t$ and $D_t$  (introduced in Section \ref{trunkz}) which 
ignore the contribution of $X$ with atypically high values. 
We extend this approach, defining also a truncated version $M^{(q)}_\gep$ for  $M_\gep$.
Proposition \ref{tiptop} justifies this procedure by showing that for large $q$, with high probability, 
$M^{(q)}_t$ coincides with $M_t$. We present in Appendix \ref{tekos1} a  proof of Proposition \ref{tiptop}. It relies on standard Gaussian tools  and classic ideas developed for the study of the maximum of $\log$-correlated fields  and branching random walk see for e.g.\ \cite{acosta, AidekonShisimple}, but is much shorter than what has appeared so far in the literature.

\medskip

The next steps are based on a couple of novel ideas. 
The key point is Lemma \ref{lecondit} which relates $M^{(q)}_t$ to $D^{(q)}_t$ using conditional expectation. First, this allows to transfer second moment estimates from $M^{(q)}_t$ to $D^{(q)}_t$,
to show that $D^{(q)}_t$ is bounded in $L^2$ (see the proof of Proposition \ref{convdq} in Section \ref{dddt}) and hence that $D_t$ converges.
More importantly, with very little additional efforts, we can use this show to that $\sqrt{\pi t/2}M^{(q)}_t$  
is close to $D^{(q)}_t$ in $L^2$ for large values of $t$ and hence that both sequences 
converge in probability to the same limit, see Section \ref{mmmt}.
In Section \ref{mmmgep} and \ref{subkrit} using an analogue strategy (using Lemma \ref{lecondit2}/\ref{conditorei} instead of Lemma \ref{lecondit}), we prove that  $M_\gep$ and $M^{\alpha}$ (appropriately normalized) converge to  $D_\infty$.

\subsection{Critical chaos with a different reference measure}\label{deesk}

In the definition of $M_\gep$, $M_t$ and $D_t$ \eqref{forboundedfunctions}, \eqref{defalfat} and \eqref{laderivvv}, it is possible replace the Lebesgue measure on $\bbR^d$ with an arbitrary locally finite measure $\mu$ by setting 
\begin{equation}
M_\gep := \int_{\bbR^d} f(x) e^{\sqrt{2d} X_{\gep}(x)- d K_{\gep}(x) } \mu( \dd  x).
\end{equation}
(and similarly for $M_t$ and $D_t$), and ask whether the sequences  $\sqrt{\pi\log (1/\gep)/2}M_\gep$  $\sqrt{\pi t/2} M_t$ and $D_t$ converge weakly in that setup.
It turns out that the method we exposed below can give an answer to the question slightly beyond the case of Lebesgue measure.
More precisely, all our proofs (except that of \eqref{gageli} where one needs to modify the gage function $\varphi$) extend without modification to the case where $\mu$ is a nonnegative measure that satisfies
\begin{equation}\label{lacondition}
 \forall x\in \bbR^d, \quad  \int_{B(x,e^{-e})^2} \frac{\left(\log \log \frac{1}{ |x-y|}\right)\mu(\dd x) \mu(\dd y)  }{|x-y|^d \left(\log  \frac 1 {|x-y|}\right)^{3/2}}<\infty,
\end{equation}
with $\gamma\in (0,1/2)$.
The above condition originates from Proposition \ref{tiptop}, more precisely it is required to deduce \eqref{whendelta} from \eqref{tek1}-\eqref{tek4} (the necessity for the $\log \log$ term in the numerator comes from the second term in  \eqref{joke3})

\medskip

Not that  while \eqref{lacondition} allows to consider measures that are singular w.r.t.\ to Lebesgue measure, it still requires that the measure to be in a certain sense  almost as  ``spread out'' as the Lebesgue measure. It is worth noting for instance that \eqref{lacondition} is not satisfied if $d=2$ and $\mu=\mu_T$ is taken to be the occupation measure of a two dimensional Brownian Motion $(B_t)_{t\in [0,T]}$. Defining the critical GMC over this occupation measure is the first step of the construction of Liouville Brownian Motion at criticality \cite{lioucrit}.
%
%
%

\medskip

\noindent {\bf Acknowledgement: } The author is gratefull to Ellen Powell and Rémi Rhodes for their insightful comments on a first draft of the manuscript. This work was supported by a productivity grant from CNPq and a JCNE grant from FAPERJ.

\section{Technical preliminaries}\label{prelim}

\subsection{Gaussian tools}

Let us recall here a couple of basic results concerning Gaussian processes.
Firstly, the Cameron-Martin formula, that indicates how the distribution of a Gaussian field is affected by an exponential tilt.

\begin{proposition}\label{cameronmartinpro}

 Let $(Y(z))_{z\in \cZ}$ be a centered Gaussian field indexed by a set $\cZ$. We let  $H$ denote its covariance and $\bP$ denote its law. 
Given $z_0\in \cZ$ let us define $\tilde \bP_{z_0}$ the probability obtained from $\bP$ after a tilt by $Y(z_0)$ that is
\begin{equation}
 \frac{\dd \tilde \bP_{z_0}}{\dd \bP}:= e^{Y(z_0)- \frac{1}{2} H(z_0,z_0)}
\end{equation}
Under $\tilde \bP_{z_0}$, $Y$ is a Gaussian field with covariance $H$,
and mean  $\tilde \bbE_{z_0}[ Y(z)]=H(z,z_0).$
\end{proposition}

\noindent Secondly, bounds on the probability for a Brownian Motion to remain below a  threshold.
The first two estimates are direct consequences of the reflection principle, the third one is obtained by applying  the optional stopping Theorem to the martingale $e^{2aB_t-2a^2 t}$. 

\begin{lemma}\label{stupid}
 
 Let $B$ be a standard Brownian Motion and let $\bP$ denote its distribution. Setting $ \mathfrak g_{t}(a):= \int^{u_+}_0 e^{-\frac{z^2}{2t}} \dd z,$
 we have
 \begin{equation}
  \bP\left[ \sup_{s\in[0,t]} B_s \le a \right]=\sqrt{ \frac{2\pi}{t}}\mathfrak g_{t}(a)\le \sqrt{\frac{2\pi}{t}}a.
 \end{equation}
If $\bP_{a}[ \cdot | \ B_t=b]$ denote the distribution of the Brownian bridge of length $t$ starting from $a$ and ending at $b$
with $ab\ge 0$ then
\begin{equation}\label{bridjoux}
 \bP_{a}[ \forall s\in[0,t], \ B_s\ge 0 \  | \ B_t=b]= \left(1-e^{\frac{2ab}{t}}\right) \le  1\wedge \frac{2ab}{t}. 
\end{equation}
If $a, b> 0$ then,
\begin{equation}
 \bP[\forall t>0, \ B_t\le at+b]=1-e^{-2ab}\le 2ab.
\end{equation}

\end{lemma}

\subsection{Simple observations concerning our Gaussian fields}\label{considerations}

We compile here a list  of useful facts concerning the covariance and dependence structure of $X_t$ and $X_{\gep}$.
We introduce a third field which appears when considering conditional expectations of $M_{\gep}$. We let $X_{t,\gep}$ denote the mollification of $X_t$
\begin{equation}
 X_{t,\gep}(y)= \int_{\bbR^d} X_t(y)\theta_{\gep}(x-y) \dd y= \bbE\left[ X_{\gep} \ | \ \cF_t \right].
\end{equation}
By construction the increments of $X_t$ are orthogonal in $L^2$ and hence independent. 
Setting
\begin{equation}\label{barxx}
 X^{(s)}_t(x)= X_{t+s}(x)-X_s(x) \text{ and } \bar X_t(x)= X^{(0)}_t(x),
\end{equation}
the field $(\bar X_{t}(x))_{x\in \bbR^d, t\ge 0}$ has covariance $\bar K_{s\wedge t}(x,y)$ (recall \eqref{kkttt}). 
In particular, since $\bar K_t(x,x)=t$, this implies that for every fixed $x\in \bbR^d$ and $s\ge 0$,
$t\mapsto  X^{(s)}_t(x)$ is a standard Brownian Motion independent of $\mathcal F_s$.

\medskip

\noindent Furthermore, recalling \eqref{defqt}, we necessarily have   $t' \in [t,t+\eta_1/\eta_2]$. Our assumption that $\kappa$ is supported on $B(0,1)$ implies that $Q_u(x,y)=0$ if  $|x-y|\ge e^{-u}$.
In particular  if $|x-y|\ge e^{-s}$, 
  $X^{(s)}_{\cdot}(x)$ and  $X^{(s)}_{\cdot}(y)$ are independent Brownian Motions.
  The increments of the convoluted fields also have a finite range dependence for the same reason and $(X_{\gep}-X_{s,\gep})(x)$ is independent of  $(X_{\gep}-X_{s,\gep})(y)$ and of  $X^{(s)}_{\cdot}(y)$ when $|x-y|\ge e^{-s}+2\gep$.

  \medskip

\noindent We let $K_{t,\gep}$ and $\bar K_{t,\gep,0}$ denote the covariance of $X_{t,\gep}$  and the cross-covariance with $\bar X_t$, 
\begin{equation}\label{crossover}\begin{split}
 K_{t,\gep}(x,y)&:=\bbE[X_{t,\gep}(x)X_{t,\gep}(y)]= \int_{\bbR^d} K_t(z_1,z_2)\theta_{\gep}(x-z_1 )\theta_{\gep}(x-z_1 ) \dd z_1 \dd z_2,\\ 
 \bar K_{t,\gep,0}(x,y)&:=\bbE[X_{t,\gep}(x)\bar X_{t}(y)]=\bbE[X_{\gep}(x)\bar X_{t}(y)]
 = \int_{\bbR^d} K_t(z,y)\theta_{\gep}(x-z) \dd z.
 \end{split}
\end{equation}
Setting $\log_+ u= \max(\log u, 0)$, there exists a constant $C_R$ which is such that 
  \begin{equation}\label{thelogbound}
  \left| K_{t,\gep}(x,y) - t\wedge \log_+ \frac{1}{|x-y|\vee \gep} \right|\le C_R
  \end{equation}
The same bound are valid for $K_t$, $\bar K_t$ and $\bar K_{t,\gep,0}$, (with $\gep=0$ in the two first cases). For the upper bound on $\bar K_t$ no constant is needed and we have 
\begin{equation}\label{steak}
\bar K_t(x,y)\le t\wedge \log_+ \frac{1}{|x-y|}.
  \end{equation}
  The two bounds \eqref{thelogbound}-\eqref{steak} are easily obtained from the definition \eqref{crossover} and \eqref{defqt}. We include a proof in Appendix \ref{prosteak} for completeness.

\subsection{Truncation of $M_t$, $M_\gep$  and $D_t$}\label{trunkz}
We introduce here  truncated versions of our processes that are the starting point of all our proofs. 
Given $q>0$, and $R>0$, we set 
\begin{equation}\begin{split}\label{aqx}
 A^{(q)}_t(x)&:= \{ \forall s\in [0,t],\quad \bar X_s(x)< \sqrt{2d}s+q \},\\
 \cA^{(q)}_R&:= \Big\{  \sup_{t>0}\sup_{|x|\le R }\left( \bar X_t(x)- \sqrt{2d} t \right)< q\Big\}=\bigcap_{|x|\le R} \bigcap_{t>0} A^{(q)}_t(x).
\end{split}\end{equation}
Setting $t_{\gep}:=(\log 1/\gep)$,  we define  $M^{\alpha,(q)}_t$, $M^{(q)}_\gep$ and $D^{(q)}_t$ as follows 
\begin{equation}\begin{split}\label{trunkeq}
 M^{\alpha,(q)}_t(E)&:= \int_E e^{\alpha X_t(x)- \frac{\alpha^2}{2} K_t(x)}\ind_{A^{(q)}_t(x)} \dd x=: \int_E W^{\alpha,(q)}_t(x) \dd x,\\
 M^{(q)}_t(E)&:= \int_E e^{\sqrt{2d} X_t(x)- d K_t(x)}\ind_{A^{(q)}_t(x)} \dd x=: \int_E W^{(q)}_t(x) \dd x,\\
 M^{(q)}_{\gep}&:=   \int_{E}e^{\sqrt{2d}X_\gep(x)- d K_{\gep}(x)}\ind_{A^{(q)}_{t_{\gep}}(x)}\dd x =:  \int_{E}W^{(q)}_\gep(x) \dd x,\\
  D^{(q)}_t(E)&:= \int_E (\sqrt{2d} t+q-\bar X_t(x)) W^{(q)}_t(x)\dd x=: \int_E Z^{(q)}_t(x) \dd x.
\end{split}\end{equation}
We further set
 \begin{equation}\label{leslimsup}
 \bar D^{(q)}_\infty(E):=\limsup_{t\to \infty}  D^{(q)}_t(E) \quad \text{ and } \quad  \bar M^{\alpha,(q)}_\infty(E):=\limsup_{t\to \infty} \bar M^{\alpha,(q)}_t(E).
 \end{equation}
It is not difficult to check that   $D^{(q)}_t(E)$  is a  nonnegative martingale and $M^{\alpha,(q)}$ a nonnegative supermartingale. Hence almost surely the $\limsup$s in \eqref{leslimsup} are limits. 
Let us register these observations as a lemma (proved at the end of this section for completeness).
\begin{lemma}\label{martinmister}
For any fixed $E\in \cB_b$,  $D^{(q)}_t(E)$ is a nonnegative martingale for the filtration $(\cF_t)$. In particular we have almost-surely
$$  \bar D^{(q)}_\infty(E)=\lim_{t\to \infty}  D^{(q)}_t(E)<\infty.$$
For any fixed $E\in \cB_b$, $\alpha\ge 0$, $M^{\alpha,(q)}_t(E)$ is a nonnegative supermartingale for the filtration $(\cF_t)$. When $\alpha<\sqrt{2d}$, we have $\lim_{t\to \infty}M^{\alpha,(q)}_t(E)=\bar M^{\alpha,(q)}_\infty(E)$ in $L^1$.
\end{lemma}

\noindent A crucial result that validates the truncation approach is that for large values $q$ the event $\cA^{(q)}_R$ has large probability. We provide a proof in Appendix \ref{tekos1}.
\begin{proposition}\label{tiptop}
For any $R>0$ we have almost surely 
\begin{equation}
 \sup_{t>0}\sup_{|x|\le R} \left(\bar X_t- \sqrt{2d} t +\frac{\log t}{2\sqrt{2d}}-\frac{4 \log \log t}{\sqrt{2d}}\right) <\infty.
\end{equation}
In particular we have
  $\lim_{q\to \infty}  \bP\left[\cA^{(q)}_R\right]=1.$
 As a consequence given $R$ there exists a random $q_0\in \mathbb N$ such that for every $q\ge q_0$, $\alpha\in(0,\sqrt{2d}]$,  $t\ge 0$, $\gep\in(0,1]$ and $E\subset B(0,R)$
 \begin{equation}\label{coincix}
  M^{\alpha,(q)}_t(E)= M^{\alpha}_t(E) \quad \text{ and } \quad  M^{(q)}_\gep(E)= M_\gep(E).  
  \end{equation}

\end{proposition}
\noindent Let us present a couple of easy consequences of Proposition \ref{tiptop}.
\begin{cor}\label{cortiptop}
We for any $E\in \cB_b$  we have almost surely
\begin{equation}\label{mtoo}
\lim_{t\to \infty} M_t(E)=0.
\end{equation}
Given $R>0$, there exists a random $q_0\in \mathbb N$ such that a.s.\ for every $q\ge q_0$,
\begin{equation}\label{dqdt}
\lim_{t\to \infty}\sup_{E\subset R} |D_{t}(E)- D^{(q)}_{t}(E)|=0.
\end{equation}

\end{cor}
\begin{proof}[Proof of Corollary \ref{cortiptop}]
 $M_t(E)$ is a nonnegative martingale and hence it converges almost-surely. It suffices thus to show that \eqref{mtoo} holds in probability. 
 We prove that
\begin{equation}\label{l11q}
\lim_{t\to \infty} \bbE\left[ M_t(E)\ind_{\cA^{(q)}_R}\right]=0.
\end{equation}
Using  Proposition \ref{tiptop}, we see that \eqref{l11q} implies convergence in probability.
 We have 
 \begin{equation}
  \bbE\left[ M_t(E) \ind_{\cA^{(q)}_R} \right] \le  \bbE\left[ M^{(q)}_t(E)\right]= \int_E \bbE\left[  e^{\sqrt{2d}X_t(x)- dK_t(x)}\ind_{A^{(q)}_t(x)}\right] \dd x.
 \end{equation}
Using Cameron-Martin formula (Proposition \ref{cameronmartinpro}), we see that the  exponential tilt shifts $\bar X_s(x)$ by an amount $\sqrt{2d}s$
and thus using Lemma \ref{stupid} we have
\begin{equation}
 \bbE\left[  e^{\sqrt{2d}X_t(x)- dK_t(x)}\ind_{A^{(q)}_t(x)}\right]=\bbP\left[ \sup_{s\in[0,t]} \bar X_s(x) \le q\right]\le 
  \sqrt{\frac{2}{\pi t}} q.
\end{equation}
Since $E$ has finite Lebesgue measure, this allows to conclude the proof of \eqref{mtoo}. Next, using Proposition \ref{tiptop}, there exists $q_0\in \bbN$ such that $\mathcal A^{q}_R$ holds for $q\ge q_0$, and hence 
\begin{multline}
  |\bar D^{(q)}_t(E)-\bar  D_t(E)|
  =\left|\int_{E} \left(q-X_0(x)\right) e^{\sqrt{2d}X_t(x)- dK_t(x)} \dd x \right|\\ \le \left(q+ \sup_{|x|\le R} |X_0(x)|\right) M_t(B(0,R)).
\end{multline}
Since $X_0$ is continuous the prefactor is finite and hence \eqref{dqdt} follows from \eqref{mtoo}
\end{proof}

\begin{proof}[Proof of Lemma \ref{martinmister}]
 It is sufficient to prove that $Z^{(q)}_t(x)$ is a martingale for any $x$. Indeed since $\bbE[Z^{(q)}_t(x)]=q$ and $E$ is bounded, integrability conditions allow for the  exchange of 
 integral and conditional expectation. We obtain that  for $s<t$
 \begin{equation*}
  \bbE\left[  \int_E Z^{(q)}_t(x) \dd x  \ | \ \cF_s \right]=  \int_E   \bbE[    Z^{(q)}_t(x)   \ | \ \cF_s]\dd x =\int_E   Z^{(q)}_s(x)  \dd x.
  \end{equation*}
  Since $Z^{(q)}_0(x)=e^{\sqrt{2d}X_0(x)-dK_0(x)}$ is integrable and independent of $\bar Z^{(q)}_t(x):=Z^{(q)}_t(x)/  Z^{(q)}_0(x)$, it is sufficient to show that  
 $\bar Z^{(q)}_t(x)$ is a martingale.
 Since $(\bar X_t(x))_{t\ge 0}$ is an $\mathcal F_t$-Brownian motion, 
 we only have to prove that if $(B_t)_{t\ge 0}$ is a Brownian motion then 
 $$N^{(q)}_t:=(\sqrt{2d} t+q- B_t) e^{\sqrt{2d}B_t- dt}\ind_{\{  \forall s\in [0,t] ,\ B_s\le \sqrt{2d}s+q \}} $$  
 is a martingale for its natural filtration. The reader can check, using Itô formula, that $N^{(q)}_t$ is a local martingale and using a direct computation that $N^{(q)}_t\in L^2$  for every $t>0$.
To show that $M^{\alpha,(q)}_t$ is a supermartingale, we proceed similarly  and observe that 
$$ Q^{\alpha,(q)}_t:= e^{\sqrt{2d}B_t- dt}\ind_{\{  \forall s\in [0,t] ,\ B_s\le \sqrt{2d}s+q \}} $$ 
is a supermartingale. When $\alpha\in [0,\sqrt{2d})$, since $M^{\alpha,(q)}_t(E)\le M^{\alpha}_t(E)$ and  $M^{\alpha}_t(E)$ is uniformly integrable (by Theorem \ref{rappel}), so is $M^{\alpha,(q)}_t$. Hence it converges in $L_1$.
\end{proof}

\subsection{Moment estimates }
 Lemma \ref{martinmister} ensures  the convergence of $D^{(q)}_t(E)$, but  we still need to prove uniform integrability to show that the limit is nontrivial.
Our strategy is to compute the second moments. This is the purpose of the following technical estimates.
 Their proof is not conceptually difficult, but requires a couple of lengthy computation. We  postpone it  to Appendix \ref{tekos1} for this reason.
 
\begin{proposition}\label{lateknik}
Given $x,y \in \bbR^d$, $t\ge 0$ and $\gep\in (0,1)$, we set 
$$ w(x,y):=\log \frac{1}{|x-y|\wedge 1}, \quad  u(x,y,t):= w     \wedge  t, \quad    v(\gep,x,y):= w\wedge \log( 1/\gep),$$
and $t_{\gep}= \log (1/\gep)$ 
Given $q$ and $R$ positive there exists a constant $C_{q,R}$ such that the following holds for every $x,y\in B(0,R)$, $t\ge 0$, $\gep\in(0,1)$ and $\alpha\in (0,\sqrt{2d})$,

\begin{align} \label{tek1}
   \bbE\left[ W^{(q)}_t(x) W^{(q)}_t(y)\right]  &\le C e^{d u} (u+1)^{-3/2}  (t-u+1)^{-1},\\
   \label{tek2}
  \  \bbE\left[ W^{(q)}_\gep(x) W^{(q)}_\gep(y)\right] & \le C e^{d v} (v+1)^{-3/2}  ( t_{\gep}-v+1)^{-1},\\
  \label{tek3}
   \bbE\left[ W^{\alpha,(q)}_t(x) W^{\alpha,(q)}_t(y)\right] & \le C e^{d u} (u+1)^{-3/2},\\
   \label{tek4}
 \lim_{s\to \infty} \bbE\left[ W^{\alpha,(q)}_s(x) W^{\alpha,(q)}_s(y)\right] & \le C (\sqrt{2d}-\alpha)^2 e^{dw} \left( w+1\right)^{-3/2}.
\end{align}
As  consequences we have (with a possibly different constant $C_{q,R}$)
 \begin{itemize}
  \item[(i)] Setting  $\varphi(u)=u^d [1\vee (\log (1/u)]^{1/2}$,   for all $E\subset B(0,R)$ 
  \begin{equation}\label{hiphip}
  \limsup_{t\to \infty} \bbE\left[ t(M^{(q)}_t(E))^2 \right]\le C  \varphi(|E|).
  \end{equation}
  \item[(ii)] For any $E\in \mathcal B_b$, 
  \begin{equation}\label{whendelta}
  \begin{split}
  \lim_{\delta\to 0}   \sup_{t\ge 0} \int_{E^2} \bbE\left[ t W^{(q)}_t(x)W^{(q)}_t(y)\right]\ind_{\{|x-y|\le \delta\}} \dd x \dd y&=0,\\
\lim_{\delta\to 0}   \sup_{\gep\in (0,1]} \int_{E^2} \bbE\left[ \log (1/\gep) W^{(q)}_\gep(x)W^{(q)}_\gep´(y)\right]\ind_{\{|x-y|\le \delta\}} &\dd x \dd y=0,\\  \lim_{\delta\to 0} \sup_{\alpha\in(0,\sqrt{2d})}  \limsup_{t\to \infty} \int_{E^2} \bbE\left[ (\sqrt{2d}-\alpha)^{-2} W^{\alpha,(q)}_t(x)W^{\alpha,(q)}_t(y)\right]&\ind_{\{|x-y|\le \delta\}} \dd x \dd y=0.
\end{split}
  \end{equation}
  \end{itemize}
\end{proposition}

\section{The convergence of $D_t$}\label{dddt}
\noindent We  first  prove a quantitative uniform integrability statements for $D^{(q)}_t(E)$.
Using the observations of the previous section, we  deduce our first main result from it. 
\begin{proposition}\label{convdq}
For any fixed $E\in \mathcal B_b$, and $q\ge 0$, the martingale $D^{(q)}_t(E)$ is bounded in $L^2$. 
More precisely,  there exists $C_{R,q}$ such that  for any $E\subset B(0,R)$  and $t\ge 0$,
\begin{equation}\label{upup}
 \bbE\left[ D^{(q)}_t(E)^2\right] \le C_{R,q}  \varphi(|E|).
\end{equation}
\end{proposition}

\begin{proof}[Proof of Theorem \ref{derivmartin}]

Combining Lemma \ref{martinmister} and  \eqref{dqdt} in Corollary \ref{cortiptop}, 
there exists a random  $q_0\in \bbN$ such that 
for every $q\ge q_0$ and $E\subset B(0,R)$
\begin{equation}\label{zob}
 \bar D_\infty(E)=\lim_{t\to \infty}  D_{t}(E)=   \bar D^{(q)}_\infty(E).
\end{equation}
Next we prove that $\bar D_{\infty}(E)>0$ a.s.\ if $\gl(E)>0$. We first establish a zero-one law. Let $u>0$ be fixed  and recall the definition \eqref{barxx}.
Setting $s=t-u$, we have 
\begin{multline}\label{decococo}
 D_t(E)=  \int_E \left(\sqrt{2d} s-  X^{(u)}_s(x)\right) e^{\sqrt{2d}X^{(u)}_s(x)-d s} e^{\sqrt{2d}X_u(x)-d u} \dd x \\
 +\int_E  \left(\sqrt{2d} u-  \bar X_u(x)\right)e^{\sqrt{2d}X_t(x)-d t}\dd x. 
\end{multline}
The field $\bar X_u(\cdot)$ is continuous and thus bounded on $E$, recalling \eqref{mtoo} this implies that the second line term in \eqref{decococo}
tends to zero almost-surely. Since $X_u(\cdot)$ is also bounded, the last exponential factor in the first line  of \eqref{decococo} is bounded from above and below,
hence the event
 $\{\bar D_\infty(E)=0\}$  coincides with
  $\left\{ \limsup_{s\to \infty}  \int_E \left(\sqrt{2d }- X^{(u)}_s(x)\right) e^{\sqrt{2d}X^{(u)}_s(x)-d s} \dd x =0 \right\}.$
In particular, recalling Section \ref{considerations}, $\{\bar D_\infty(E)=0\}$ is  independent of  $\mathcal F_u$ for every $u>0$. Since it is measurable w.r.t.\ $\cF_{\infty}$ it is independent of itself and $\bbP[\bar D_\infty(E)=0]\in\{0,1\}$.
To prove that this last probability is equal to zero, since $\bar D^{(q)}_\infty(E)$ is increasing to
$D_{\infty}(E)$ when $q$ increases (cf. \eqref{zob}),
it is sufficient to show that $\bbP[\bar D^{(q)}_\infty(E)>0]>0$ for some value of $q$.
Using Proposition \ref{convdq}, the uniform integrability of the martingale yields
$\bbE[\bar D^{(q)}_\infty(E)]=\bbE[\bar D^{(q)}_0(E)]=q\gl(E)>0.$
\medskip

Finally we prove \eqref{lagaage}. Using \eqref{zob}, it is sufficient to prove the statement with $\bar D_t$ replaced by $\bar D^{(q)}_\infty$. For simplicity we prove it for
subsets of $[0,1]^d$ rather than of $B(0,R)$. For $n\ge 1$, we partition $[0,1]^d$ in $k_n:=2^{d 2^n}$ cube of sidelength $\delta_n=2^{-2n}$, call them $(I^{(n)}_i)^{k_n}_{i=1})$.
Given $E\in \mathcal B_b$, if $n$ is such that   $|E|\in(\delta_{n+1},\delta_n]$ then $E$ 
can be fitted in the union of at most four $I^{(n)}_i$s.
Hence it suffices to show that for all $n$ sufficiently large
\begin{equation}
\max_{i\in \lint 1, k_n\rint} \bar D^{(q)}_{\infty}(I^{(n)}_i) \le   n 2^{-n/4}.
\end{equation}
Using the union bound we obtain that 
\begin{multline}
 \bbP \left[ \max_{i\in \lint 1, k_n\rint} \bar D^{(q)}_{\infty}(I^{(n)}_i)\ge   n 2^{-n/4} \right] \le k_n \bbP\left[ \bar D^{(q)}_{\infty}(I^{(n)}_1)\ge   n 2^{-n/4}\right] \\
 \le k_n 2^{n/2} n^{-2}  \bbP\left[ \left(\bar D^{(q)}_{\infty}(I^{(n)}_1)\right)^2   \right]\le C n^{-2}. 
\end{multline}
For the last inequality, we combined  \eqref{upup} and Fatou. We conclude with Borel-Cantelli.
\end{proof}
\noindent To prove \eqref{convdq}, we deduce from our uniform bound \eqref{hiphip} on the second moment of $\sqrt{t} M^{(q)}_t$ a bound for the second moment of  $D^{(q)}_s$,  using conditional expectation.
\begin{lemma}\label{lecondit}
 Given $E\in \cB_b,$ for any fixed $s\ge 0$ we have 
 \begin{equation}\label{ziup}
  D^{(q)}_s(E)= \lim_{t\to \infty} \bbE\left[ \sqrt{\frac{\pi (t-s)}{2}} M^{(q)}_t(E) \ | \ \cF_s\right].
 \end{equation}
Furthermore, the r.h.s.\  is is nondecreasing in $t$ and convergence holds in $L^2$.
\end{lemma}

\begin{proof}
 It is sufficient to show that for any $x\in E$
 \begin{equation}\label{gross}
 Z^{(q)}_s(x)= \lim_{t\to \infty} \bbE\left[ \sqrt{\frac{\pi (t-s)}{2}} W^{(q)}_t(x) \ | \ \cF_s\right]
 \end{equation}
 monotonically.
The dominated convergence theorem used twice (first for $\int_E$ and then for $\bbE[ (D^{(q)}_s- \dots)^2]$)  implies the desired  $L^2$ convergence.
Omitting the variable $x$ for readability, factoring out  $W^{(q)}_s$ (which is $\cF_s$ measurable) and 
then using the Cameron-Martin formula for the Brownian Motion $(X^{(s)}_{\cdot}(x))$ and Lemma \ref{stupid} we obtain that
\begin{equation}\label{laconvmone}\begin{split}
  \bbE\left[  W^{(q)}_t \ | \ \cF_s\right]&= W^{(q)}_s  \bbE\left[e^{\sqrt{2d}X^{(s)}_{t-s}-d (t-s)}\ind_{\{ \forall u\in [0,t-s], \ X^{(s)}_{u} 
  \le  \sqrt{2d}(s+u)+q-\bar X_s\}} 
   \ | \ \cF_s \right]\\
\\&=W^{(q)}_s \bbP\left[  \sup_{u\in [0,t-s]} X^{(s)}_{u} \le  \sqrt{2d} s+q- \bar X_s  \ | \ \cF_s \right]\\&= W^{(q)}_s \sqrt{\frac{2}{\pi(t-s)}}  \mathfrak g_{t-s} ( \sqrt{2d} s+q- \bar X_s).
\end{split}\end{equation}
 The  result is thus a consequence of the fact that $\lim_{r\to \infty}\mathfrak g_{r}(u)=u_+$ monotonically .
\end{proof}

\begin{proof}[Proof of Proposition \ref{convdq}]
Using Lemma \ref{lecondit} we have 
 \begin{multline}
  \bbE\left[ (D^{(q)}_s(E))^2 \right]= \lim_{t\to \infty} \bbE\left[  \bbE\left[ \sqrt{\frac{\pi (t-s)}{2}} M^{(q)}_t(E) \ | \ \cF_s \right]^2 \right]
  \\ \le \limsup_{t\to \infty} \bbE\left[  \left(\sqrt{\frac{\pi (t-s)}{2}} M^{(q)}_t(E)\right)^2 \right] = \limsup_{t\to \infty} \bbE\left[  \frac{\pi t}{2} (M^{(q)}_t(E))^2  \right].
 \end{multline}
We can then conclude using the estimate \eqref{hiphip}.
\end{proof}

\begin{rem}
  The use of \eqref{ziup} in the above proof  is convenient since the moments of $\sqrt{t} M^{(q)}_t$ is slightly easier to compute than that of $D^{(q)}_t$, but the most important application of Lemma \ref{lecondit}  comes in the next section. It allows to deduce almost without efforts the convergence of $\sqrt{t} M^{(q)}_t$ (which has no martingale structure) from that of $D^{(q)}_t$.
  \end{rem}

\section{The convergence of $M_t$} \label{mmmt}

\noindent We prove Theorem \ref{critmartin} by showing that $\sqrt{\pi t/2}M^{(q)}_t$, converges to the same limit as $D^{(q)}_t$.
\begin{proposition}\label{dor}
For any $E\in \mathcal B_b$ and $q\ge 0$, we have the following convergence in $L^2$
 \begin{equation}
 \lim_{t\to \infty} \sqrt{\frac{\pi t}{2}} M^{(q)}_t(E)= \bar D^{(q)}_{\infty}(E).
 \end{equation}
 
\end{proposition}
\begin{proof}[Proof of Theorem \ref{critmartin}]
 Using \eqref{zob} and Proposition \ref{tiptop} 
  there exists a random $q_0\in \bbN$ such that for $q\ge q_0$, 
  $$\bar D^{(q)}_{\infty}(E)=\bar D_{\infty}(E) \quad \text{ and  } \quad  M^{(q)}_t(E)= M_t(E).$$ 
Using Proposition \ref{dor} we conclude that $\sqrt{\pi t/2} M_t(E)$ converges  to $\bar D_\infty(E)$ in probability.
To obtain the weak convergence of the measures, we  apply Proposition \ref{weako}.
  \end{proof}

\begin{proof}[Proof of Proposition \ref{dor}] We drop $E$ from the notation for better readability.
 Since $D^{(q)}_s$ converge in $L^2$, it is sufficient to prove that 
 \begin{equation}
 \lim_{s\to \infty} \limsup_{t\to \infty}\bbE\left[  \left|\sqrt{\frac{\pi t}{2}} M^{(q)}_t-D^{(q)}_s\right|^2 \right]=0.
 \end{equation}
Using the conditional expectation to make an orthogonal decomposition in $L^2$ we have
\begin{multline}\label{ladekomp}
 \bbE\left[  \left|\sqrt{\frac{\pi t}{2}} M^{(q)}_t-D^{(q)}_s\right|^2 \right]\\=
 \frac{\pi t}{2} \bbE\left[  \left| M^{(q)}_t-\bbE\left[  M^{(q)}_t \ | \ \cF_s \right]\right|^2 \right]
+ \bbE \! \left[  \! \left( \!
\bbE \! \left[ \! \sqrt{\frac{\pi t}{2}} M^{(q)}_t \ | \ \cF_s\right]- D^{(q)}_s \!  \right)^2 \! \right].
\end{multline}
From Lemma \ref{lecondit}, the second term in the r.h.s.\ of \eqref{ladekomp} tends to zero  when $t\to \infty$.
To control the first term,  we set  (recall \eqref{trunkeq})
$$  \xi_{s,t}(x):= W^{(q)}_{t}(x)- \bbE\left[ W^{(q)}_{t}(x) \ | \ \cF_s \right].$$
Expanding the square we obtain
\begin{equation}\label{xixi}
\bbE\left[  \left| M^{(q)}_t-\bbE\left[  M^{(q)}_t \ | \ \cF_s \right]\right|^2 \right] =\int_{E^2} \bbE\left[ \xi_{s,t}(x) \xi_{s,t}(y)   \right] \dd x \dd y.
\end{equation}
Recalling Section \ref{considerations} (in particular the definition \eqref{barxx})
we note that $\xi_{s,t}(x)$ is measurable with respect to $\cF_s\vee \sigma( (X^{(s)}_u(x))_{u\ge 0})$.
        If $|x-y|\ge e^{-s}$ then  $X^{(s)}_{\cdot}(x)$ and $X^{(s)}_{\cdot}(y)$  are  independent and both processes are independent of $\cF_s$. This implies the  conditional independence of $\xi_{s,t}(x)$ and  $\xi_{s,t}(y)$ given $\cF_s$ and hence
\begin{equation}\label{czer}
\bbE\left[ \xi_{s,t}(x) \xi_{s,t}(y) \ | \ \cF_s  \right]= \bbE\left[ \xi_{s,t}(x) \ | \ \cF_s  \right] \bbE\left[ \xi_{s,t}(y)  \ | \ \cF_s \right]=0.
\end{equation}
On the other hand when $|x-y|\le e^{-s}$, we have 
\begin{equation}\label{cpzer}
 \bbE\left[ \xi_{s,t}(x) \xi_{s,t}(y)\right]\le 
  \bbE\left[ W^{(q)}_{t}(x) W^{(q)}_{t}(y)\right].
\end{equation}
Thus combining \eqref{czer} and \eqref{cpzer} we have
\begin{equation}\label{czer3}
\bbE\left[  \left| M^{(q)}_t-\bbE\left[  M^{(q)}_t \ | \ \cF_s \right]\right|^2 \right]
\le \int_{E^2} \bbE\left[ W^{(q)}_{t}(x) W^{(q)}_{t}(y)   \right] \ind_{\{|x-y|\le e^{-s}\} } \dd x \dd y.
\end{equation}
 Using Equation \eqref{whendelta} (first line) from  Proposition \ref{lateknik}, we obtain that 
 \begin{equation}
   \lim_{s\to \infty} \limsup_{t\to \infty}\frac{\pi t}{2} \bbE\left[  \left| M^{(q)}_t-\bbE\left[  M^{(q)}_t \ | \ \cF_s \right]\right|^2 \right]=0.
   \end{equation}
Recalling \eqref{ladekomp} this concludes the proof.
\end{proof}

\section{The convergence of $M_{\gep}$}\label{mmmgep}

\noindent The strategy of the previous section can be adapted to prove the convergence of $M^{}_{\gep}$. 
We  show that $\sqrt{{\pi \log(1/\gep)}/{2}} M^{(q)}_\gep(E)$ converges to the same limit as $D^{(q)}_t(E)$ and $M^{(q)}_t(E)$.

\begin{proposition}\label{gepqgep}
 For any $E\in \mathcal B_b$ and $q\ge 0$, we have the following convergence in $L_2$
 \begin{equation}
 \lim_{t\to \infty} \sqrt{\frac{\pi \log(1/\gep)}{2}} M^{(q)}_\gep(E)= \bar D^{(q)}_{\infty}(E)
 \end{equation}
\end{proposition}

\begin{proof}[Proof of Theorem \ref{mainres}]
A first important observation is  that the veracity of the statement ``$M^{}_\gep(E)$ converges in probability to the same limit for all choices of mollifier''  only depends the distribution of the process $X$ (recall that this is a process indexed by $C_c(\bbR^d)$ with covariance given by \eqref{hatK}).
Hence without loss of generality, it is sufficient to prove it in the case were the  probability space contains the martingale approximation sequence $(X_t)$ described in Section \ref{martindecoco}.
In that case, combining  Propositions \ref{tiptop} and \ref{gepqgep} in the same way as in the proof of Theorem \ref{critmartin}, we obtain \eqref{inproba} with $M'(E)=\bar D_{\infty}(E)$.
In particular the limit does not depend on the mollifying kernel.
\end{proof}
\noindent The proof of Proposition \ref{gepqgep} is very similar to that of Proposition \ref{dor} but we need the following replacement for Lemma \ref{lecondit}.
\begin{lemma}\label{lecondit2}
  For any fixed $s\ge 0$ we have the following convergence in $L^2$.
 \begin{equation}
  D^{(q)}_s(E)= \lim_{\gep \to 0} \bbE\left[ \sqrt{\frac{\pi (\log (1/\gep))}{2}} M^{(q)}_\gep(E) \ | \ \cF_s\right].
 \end{equation}
 
\end{lemma}
\begin{proof}
As there is no monotonicity in $\gep$, the proof of the $L^2$ convergence is more technical than that of Lemma \ref{lecondit}.
Using the convention $t=t_{\gep}=(\log 1/\gep)$, we need to show that
\begin{equation}\label{gregre}
 \lim_{\gep \to 0} \sup_{x\in E} \bbE\left[ \left( Z^{(q)}_s(x)- \bbE\left[ \sqrt{\frac{\pi t}{2}} W^{(q)}_\gep(x) \ | \ \cF_s\right]\right)^2 \right]=0.
\end{equation}
The uniformity in \eqref{gregre} implies that the convergence in $L^2$ is maintained after integrating w.r.t.\ $x$ over $E$. Taking $\cF_s$-measurable 
terms out of the expectation we obtain
\begin{equation}
   \bbE[  \sqrt{{\pi t}/{2}}  W^{(q)}_\gep(x) \ | \ \cF_s ] = W^{(q)}_{s,\gep}(x) \times  V^{(q)}_{s,\gep}(x),
 \end{equation}
 where
 \begin{equation}\begin{split}
 W^{(q)}_{s,\gep}(x)&:= e^{\sqrt{2d}X_{s,\gep}(x)-d K_{s,\gep}(x)} \ind_{A^{(q)}_s(x)},\\
  V^{(q)}_{s,\gep}(x)&:=   \sqrt{\frac{\pi t }{2}} \bbE\left[ e^{\sqrt{2d}(X_{\gep}-X_{s,\gep})(x)- d(K_{\gep}-K_{s,\gep})(x)}
  \ind_{\{\forall u\in [s,t],\  \bar X_u(x)\le q+ \sqrt{2d}u \}   } \ | \ \cF_s \right].
   \end{split}\end{equation}
 In view of the inequality 
 $\bbE[ (A_tB_t-AB)^2]^2\le 4 \left(\bbE[(A_t-A)^4]\bbE[B^4_t] +  \bbE[(B_t-B)^4]\bbE[ A^4] \right)$
 (valid for arbitrary random variables $A,B, A_t$ and $B_t$),
to prove  \eqref{gregre}, it is sufficient to prove the two following convergences,
and that $W^{(q)}_s(x)$ and $(q+\sqrt{2d}s -\bar X_s(x))$ are uniformly bounded in $L^4$ when $x$ varies 
\begin{equation}\label{unil4}\begin{split}
 \lim_{\gep \to 0}\sup_{x\in E} \ & \bbE\left[ (W^{(q)}_{s,\gep}(x)-W^{(q)}_s(x))^2\right]=0,\\
 \lim_{\gep \to 0}\sup_{x\in E} \ & \bbE\left[ (V^{(q)}_{s,\gep}(x)- (q+\sqrt{2d}s -\bar X_s(x))_+)^2 \right]=0.
  \end{split}
\end{equation}
The first line in \eqref{unil4} follows from the the uniform convergence of $K_{s,\gep}(x,y)$  and $K_{s,\gep,0}(x,y)$  towards $K_{s}(x,y)$, which implies that
\begin{equation}
 \lim_{\gep\to 0} \sup_{x\in E}\bbE\left[  \left(e^{\sqrt{2d}X_{s,\gep}(x)-d K_{s,\gep}(x)}- e^{\sqrt{2d}X_{s}(x)-d K_{s}(x)}\right)^4 \right]=0.
\end{equation}
For the second  line in \eqref{unil4}, by translation invariance of $\bar X$, the quantity in the supremum does not depend on $x$. 
Thus we do not need to worry about the $\sup$ and omit $x$ in the notation.  We rewrite the event appearing in the definition of   $V^{(q)}_{s,\gep}$ as
$$\{\forall u\in [0,t-s],\   X^{(s)}_u\le q+ \sqrt{2d}(s+u)- \bar X_s \}.$$
Using the Cameron-Martin formula (Proposition \ref{cameronmartinpro}), 
the exponential tilt in  $f^{(q)}_{s,\gep}$  has the effect of shifting the mean of $X^{(s)}_{u}(x)$ by an amount (recall \eqref{crossover})
 $$\bbE[ (X_\gep-X_{s,\gep})(x)\bar X^{(s)}_u(x)]= K_{s+u,\gep,0}(x)-K_{s,\gep,0}(x) =:K^{(s)}_{u,\gep,0}.$$ 
 Hence we have
\begin{equation}\label{bbbb}
  V^{(q)}_{s,\gep}\!
 =\!\sqrt{\frac{\pi t}{2}}\bbP\left[ \forall u \in [0,t-s], \  \bar X^{(s)}_u\le \sqrt{2d}\left( \! u-K^{(s)}_{u,\gep,0}\right)+(q+\sqrt{2d}s- \bar X_s) \ | \ \cF_s\right].
 \end{equation}
Since
 $K^{(s)}_{u,\gep,0}= \int^{s+u}_s \left(\int_{\bbR^d}Q_v(0,z)\theta_{\gep}(z) \dd z\right) \dd v$
and $Q_v(0,z)\le 1$, we have $K^{(s)}_{u,\gep,0}\le u$. Injecting  this  bound in \eqref{bbbb}, we obtain from Lemma \ref{stupid} 
\begin{equation}\label{dnz}
 V^{(q)}_{s,\gep}
\ge \sqrt{\frac{t}{(t-s)}}\mathfrak g_{t-s}(q+\sqrt{2d}s- \bar X_s).
\end{equation}
To obtain a bound in the other direction, recalling \eqref{defqt} and the regularity of $\kappa$ we obtain that for some universal constant $C$ we have
$(1-Q_{v}(0,z))\le C(e^{v}z)^2$.
Setting $\bar t_{\gep}= \log(1/\gep)-\sqrt{\log 1/\gep}$, we have for any $\gep$ sufficiently large and $u\le \bar t-s$ 
 \begin{multline}
  (u-K^{(s)}_{u,\gep,0})= \int^{s+u}_s \left(\int_{\bbR^d}(1-Q_v(0,z))\theta_{\gep}(z) \dd z\right) \dd v\\ \le C \int^{u+s}_{u} e^{2v}\gep^2 \dd v 
   \le (C/2) e^{2\bar t}\gep^2 \le  \frac{\delta_{\gep}}{\sqrt{2d}}, 
 \end{multline}
 where $\delta_{\gep}= C \sqrt{d/2}  e^{-\sqrt{\log (1/\gep)}}$.
Hence we have 
 \begin{multline}\label{upz}
 V^{(q)}_{s,\gep} \le \sqrt{\frac{\pi t}{2}}\bbP\left[ \forall u \in [0,\bar t-s], \  \bar X^{(s)}_u\le \delta_{\gep}+q+\sqrt{2d}s- \bar X_s \ | \ \cF_s\right]
\\ 
\le  \sqrt{\frac{t}{(\bar t-s)}} \mathfrak g_{\bar t-s}(q+\sqrt{2d}s- \bar X_s+\delta_{\gep}).
\end{multline} 
Both upper \eqref{upz} and lower bound \eqref{dnz} converge to  $(q+\sqrt{2d}s- \bar X_s)_+$ in $L^4$ as   $\gep\to 0$ which conclude the proof of \eqref{unil4} and hence of the lemma.
 \end{proof}

\begin{proof}[Proof of Proposition \ref{gepqgep}]
 Proceeding like  for Proposition \ref{dor}, 
 it is sufficient to prove that 
 \begin{equation}
 \lim_{s\to \infty} \limsup_{\gep\to 0}\bbE\left[  \left|\sqrt{\frac{\pi (\log 1/\gep)}{2}} M^{(q)}_\gep-D^{(q)}_s\right|^2 \right]=0.
 \end{equation}
 After a decomposition like \eqref{ladekomp} and using Lemma \ref{lecondit2}, we only need to prove the following 
 \begin{equation}\label{lllast}
 \lim_{s\to \infty} \limsup_{\gep\to 0} \log  (1/\gep) \bbE\left[  \left| M^{(q)}_\gep-\bbE\left[  M^{(q)}_\gep  \ | \ \cF_s \right] \right|^2 \right]=0.
 \end{equation}
Setting $
 \xi_{s,\gep}(x) := W^{(q)}_{\gep}(x)-\bbE[W^{(q)}_{\gep}(x) \ | \ \cF_s]$ (recall \eqref{trunkeq})
we have like for \eqref{xixi}
\begin{equation}\label{xixi2}
\bbE\left[  \left| M^{(q)}_\gep-\bbE\left[  M^{(q)}_\gep \ | \ \cF_s \right]\right|^2 \right] 
=\int_{E^2} \bbE\left[ \xi_{s,\gep}(x) \xi_{s,\gep}(y)   \right] \dd x \dd y.
\end{equation}
Recalling the observations of Section \ref{considerations}, 
$(X^{(s)}_{\cdot}(x), (X_{\gep}-X_{\gep,s})(x) )$ and  $(X^{(s)}_{\cdot}(y), (X_{\gep}-X_{\gep,s})(y))$ are independent and independent of $\cF_s$ whenever
$|x-y|\ge e^{-s}+2\gep$. Hence  
 $\xi_{s,\gep}(x) $ and $\xi_{s,\gep}(y) $ are conditionally independent given $\cF_s$ and   like for \eqref{czer} we have
$$ |x-y|\ge e^{-s}+2\gep \quad \Rightarrow  \quad  \bbE\left[\xi_{s,\gep}(x) \xi_{s,\gep}(y) \right]=0.$$
Proceeding as in  \eqref{czer3}, this implies 
 \begin{equation}
\bbE\left[  \! \left| M^{(q)}_\gep-\bbE\left[  M^{(q)}_\gep \ | \ \cF_s \right]\right|^2  \! \right]
  \le \int_{E^2} \!\! \!  \bbE\left[ W^{(q)}_{\gep}(x)  W^{(q)}_{\gep}(y) \right] \ind_{\{|x-y|\le e^{-s}+2\gep\}} \dd x \dd y.
 \end{equation}
Using Equation \eqref{whendelta} (second line) we conclude the proof  of  \eqref{lllast}.
\end{proof}

\section{The convergence of subcritical chaos}\label{subkrit}

\noindent As in the previous section, we establish the convergence for the truncated version of $M^{\alpha}$. 
\begin{proposition}\label{stick}
For any $E\in \mathcal B_b$ we have the following convergence in $L^2$
\begin{equation}
 \lim_{\alpha \uparrow \sqrt{2d}} \frac{\bar M^{\alpha,(q)}_{\infty}(E)}{\sqrt{2d}-\alpha}=2\bar D^{(q)}_{\infty}(E)
\end{equation}

\end{proposition}

\begin{proof}[Proof of Theorem \ref{ziklob}]
As in the proof  Theorem \ref{mainres}, we first observe that the event ``$(\sqrt{2d}-\alpha)M^{\alpha}(E)$ and $\sqrt{2 \pi (\log 1/\gep)}M^{\sqrt{2d}}_{\gep}(E)$ converge towards the same limit'' only depends on the distribution of the process $X$. Hence, we can focus on the case where the probability space includes a martingale approximation sequence (the framework of Proposition \ref{stick}). To conclude we observe that 
from Proposition \ref{tiptop} (taking the limit of \eqref{coincix}) and \eqref{zob}, there exists a random  $q_0$ such that for $q\ge q_0$ we have  for every $\alpha\in (0,\sqrt{2d})$
$ \bar M^{\alpha,(q)}_{\infty}(E)=M^{\alpha}(E)$ and $\bar D^{(q)}_{\infty}(E)=  \bar D_{\infty}(E)$.
\end{proof}

\noindent To prove Proposition \ref{stick} we use the following replacement for Lemma \ref{lecondit}.

\begin{lemma}\label{conditorei}
Given $E\in \mathcal B_b(E)$, and $s\ge 0$ the following convergences holds in $L^2$.
  \begin{equation}\label{trickle}
 \bbE\left[ \bar M^{\alpha,(q)}_{\infty}(E) \ | \ \cF_s\right]
  =\lim_{t\to \infty}\bbE\left[ M^{\alpha,(q)}_{t}(E) \ | \ \cF_s\right],
 \end{equation}
\begin{equation}\label{tooproove}
\lim_{\alpha \uparrow \sqrt{2d}} \bbE\left[ \frac{\bar M^{\alpha,(q)}_{\infty}(E)}{\sqrt{2d}-\alpha} \ | \ \cF_s\right]=2 D^{(q)}_{s}(E).
\end{equation}

\end{lemma}

 \begin{proof}
Since  $M^{\alpha,(q)}_{t}(E)$ converges in $L^1$ (cf. Lemma \ref{martinmister}),
we already know that \eqref{trickle} holds in $L^1$.
Furthermore since the sequence $\bbE\left[ M^{\alpha,(q)}_{t}(E) \ | \ \cF_s\right]$ is decreasing (from the supermartingale property) and since  $M^{\alpha,(q)}_{s}(E)\in L^2$, the convergence also holds in $L^2$ by dominated convergence. We now move to the proof of \eqref{tooproove}. 
Exchanging the integral over $E$ with conditional expectation in \eqref{trickle} we obtain 
\begin{equation}\label{ouane}
\bbE\left[\bar M^{\alpha,(q)}_{\infty}(E) \ | \ \cF_s\right]= \lim_{t\to \infty}\int_E  \bbE\left[ W^{\alpha,(q)}_t(x) \ | \ \cF_s \right] \dd x
\end{equation}
 Dropping $x$ in the notation for readability we obtain from Cameron-Martin formula that
 \begin{equation*}
  \bbE\left[ W^{\alpha,(q)}_t \ | \ \cF_s \right]= 
  W^{\alpha,(q)}_s \bbP\left[ \forall u\in  [0,t-s], \  X^{(s)}_u
  \le q+ \sqrt{2d}s -\bar X_s +  (\sqrt{2d}-\alpha)u \ | \ \cF_s\right].
 \end{equation*}
The r.h.s.\ is nondecreasing in $t$, and from Lemma \ref{stupid}, we have 
\begin{equation}
 \lim_{t\to \infty}   \bbE\left[ W^{\alpha,(q)}_t(x) \ | \ \cF_s \right]=   W^{\alpha,(q)}_s(x) \left( 1- e^{-2(\sqrt{2d}-\alpha) [q+ \sqrt{2d}s -\bar X_s(x)]}\right).
\end{equation}
Using \eqref{ouane} and dominated convergence (our integrand is bounded by $W^{\alpha,(q)}_s(x)$) we obtain
 \begin{equation}\label{trickle2}
\bbE\left[ \frac{\bar M^{\alpha,(q)}_{\infty}(E)}{{\sqrt{2d}-\alpha}} \ | \ \cF_s\right]
=\int_E
W^{\alpha,(q)}_s(x) \left(\frac{1- e^{-2(\sqrt{2d}-\alpha) [q+ \sqrt{2d}s -\bar X_s(x)]}}{\sqrt{2d}-\alpha}\right)\dd x.
 \end{equation}
The integrand on the r.h.s.\ converges to $2 D^{(q)}_s(x)$ when $\alpha \uparrow \sqrt{2d}$. It is also smaller than   $|\sqrt{2d} s+q-\bar X_s(x)|e^{\sqrt{2d}|X_s(x)|}$ for every $\alpha \in (0,\sqrt{2d})$. 
We  conclude the proof of \eqref{tooproove} using dominated convergence. The domination also ensures boundedness in $L^p$ for  $p>2$ and thus  the convergence holds in $L^2$.
 \end{proof}

 \begin{proof}[Proof of Proposition \ref{stick}]

   Proceeding like  for Propositions \ref{dor} and \ref{gepqgep}, using first approximation of $\bar D^{(q)}_\infty$ with $D^{(q)}_s$ then an orthogonal decomposition conditioning to $\mathcal F_s$ and finally  Lemma \ref{conditorei} we are left with proving 

 \begin{equation}\label{lllastt}
 \lim_{s\to \infty} \limsup_{\alpha \uparrow \sqrt{2d}} (\sqrt{2d}-\alpha)^2 \bbE\left[  \left| \bar M^{\alpha,(q)}_\infty(E)-\bbE\left[  \bar M^{\alpha,(q)}_\infty(E)  \ | \ \cF_s \right] \right|^2 \right]=0.
 \end{equation}
 We omit the dependence in $E$ for readability. To apply our usual scheme of proof we want to replace $\infty$ by $t$ in \eqref{lllastt}. Fatou implies that 
 $ \bbE\left[ (\bar M^{\alpha,(q)}_\infty)^2\right]\le \liminf_{t\to \infty} \bbE\left[ (\bar M^{\alpha,(q)}_t)^2\right],$ combined with  \eqref{trickle} this yields 
\begin{equation}\label{important}
\bbE\left[  \left| \bar M^{\alpha,(q)}_\infty-\bbE\left[  \bar M^{\alpha,(q)}_\infty  \ | \ \cF_s \right] \right|^2 \right]\le
\liminf_{t\to \infty}\bbE\left[  \left|  M^{\alpha,(q)}_t-\bbE\left[   M^{\alpha,(q)}_t  \ | \ \cF_s \right] \right|^2 \right].
\end{equation}
Setting
$ \xi^{\alpha}_{s,t}(x):=W^{\alpha,(q)}_t(x)- \bbE\left[W^{\alpha,(q)}_t(y) \ | \ \cF_s \right]$
and repeating the argument \eqref{xixi}-\eqref{czer3} we obtain 
\begin{multline}\label{torche}
 \bbE\left[  \left|  M^{\alpha,(q)}_t-\bbE\left[   M^{\alpha,(q)}_t  \ | \ \cF_s \right] \right|^2 \right] =\int_{E^2}\bbE[\xi^{\alpha}_{s,t}(x)\xi^{\alpha}_{s,t}(y) ]\dd x \dd y\\ 
 =\int_{E^2}\!\!\!\! \ind_{\{|x-y|\le e^{-s}\}}\bbE[\xi^{\alpha}_{s,t}(x)\xi^{\alpha}_{s,t}(y) ]\dd x \dd y =\int_{E^2} \!\!\!\! \ind_{\{|x-y|\le e^{-s}\}}\bbE[W^{\alpha,(q)}_t\!(x) W^{\alpha,(q)}_t\!(y) ]\dd x \dd y.
\end{multline}
Combining \eqref{lllastt}, \eqref{important} and \eqref{torche}, we can conclude using the last line of \eqref{whendelta}.
 \end{proof}

 \appendix

 \section{Proof of technical estimates}\label{teklesproofs}

\subsection{Proof of Proposition \ref{tiptop}}\label{tekos1}

\noindent For simplicity we replace $B(0,R)$ by $[0,1]^d$ but this entails no loss of generality.
Like in \cite{MR3262492}, the first idea is to restrict ourselves to a discrete set instead of checking the inequality at every coordinates. We define
\begin{equation}\begin{split}
I_n:= \{ {\bf i}\in \bbZ^d  \ : \ {\bf i} e^{-n} \in [0,1]^d\}, \quad  &A_{n,{\bf i}}:= [n,n+1)\times \left(  {\bf i} e^{-n}+ [0,e^{-n})^d \right),\\
Y_{n,{\bf i}}:= \bar X_{n}( {\bf i}e^{-n}) , \quad 
&Z_{n,{\bf i}}:= \max_{(t,x)\in A_{n,{\bf i}}} \left(\bar X_{t}(x)-Y_{n,{\bf i}}\right).
\end{split}\end{equation}
The statement we want to prove can be restated as follows
\begin{equation}\label{dabest}
 \sup_{n\ge 0} \max_{{\bf i}\in I_n} \left[ Y_{n,{\bf i}}+Z_{n,{\bf i}} - \sqrt{2d}n + \frac{\log n}{2\sqrt{2d}}- \frac{4(\log \log n)}{\sqrt{2d}}\right]<\infty. 
 \end{equation}
The following estimate (whose proof we postpone to the end of the section) ensures that the tails of $Z_{n,{\bf i}}$ are uniformly subgaussian.
\begin{lemma}\label{ggttail}
 There exists a constant $c>0$ such that for all $n\ge 0$, ${\bf i}\in \mathbb Z^d$ and $\gl\ge 0$
 \begin{equation}\label{kilaro}
 \bbP[Z_{n,{\bf i}}\ge \gl ] \le 2 \exp\left( - c \gl^2  \right).
\end{equation}
\end{lemma}
\noindent Combining \eqref{kilaro} and the fact that $Y_{n,{\bf i}}$ is a centered Gaussian with variance $n$, we obtain
\begin{equation}\label{convovo}
 \bbP\left[ Y_{n,{\bf i}}+Z_{n,{\bf i}} \ge A \right] \le \frac{2}{\sqrt{2\pi n}}\int_{\bbR}  \exp\left( - \frac{u^2}{2n} - c (A-u)_+^2  \right) \dd u.
\end{equation}
This implies  that 
\begin{equation*}
\sum_{n\ge 0}  \sum_{i \in I_n}\bbP \left[ Y_{n,{\bf i}}+Z_{n,{\bf i}} - \sqrt{2d}n - \frac{\log n}{\sqrt{2d}}\ge  0\right] \le \sum_{n\ge 0}  
\sum_{i \in I_n} Cn^{-3/2} e^{- dn}\le C' \sum_{n\ge 0}  n^{-3/2} <\infty.
\end{equation*}
Hence using Borel-Cantelli we obtain the following rough bound
\begin{equation}\label{firstbc}
 \sup_{n\ge 0} \sup_{i \in I_n} \left[ Y_{n,{\bf i}}+Z_{n,{\bf i}}- \sqrt{2d}n + \frac{\log n }{\sqrt{2d}} \right]<\infty.
\end{equation}
We now build on \eqref{firstbc} to obtain \eqref{dabest}. We set for $q\in \bbN$
$$\mathcal B_q:=  \bigcap_{n\ge 0} \bigcap_{i \in I_n}\left\{ Y_{n,{\bf i}}+Z_{n,{\bf i}}\le  \sqrt{2d}n + \frac{\log n}{\sqrt{2d}}+q  \right\} $$
From \eqref{firstbc} we have
$ \lim_{q\to \infty} \bbP\left[ \mathcal B_q \right]=1$.
Using the bound \eqref{bridjoux} together with the fact that 
$$ \mathcal B_q \subset \left\{ \forall t\in [0,n], \ \bar X_t({\bf i}e^{-n})\le  \sqrt{2d}t + \frac{\log n}{\sqrt{2d}} +q\right\} $$
we can compute an upper bound on the density distribution of $Y_{n,{\bf i}}$ restricted to $\cB_q$. Using improper but unambiguous notation we have
\begin{multline}
\bbP[Y_{n,{\bf i}}\in \dd u  \cap \cB_q] \le  \bP\left[ B_n \in \dd  u \ ; \ \forall t\in [0,n], B_t\le   \sqrt{2d}t + \frac{\log n}{\sqrt{2d}}+q\right]
\\ \le 
\frac{\dd u}{\sqrt{2\pi n}} e^{\frac{-u^2}{2n}} (q+\log n)(q+\log n+\sqrt{2d} n- u)_+.
\end{multline}
Similarly to \eqref{convovo}, setting $A_n:=  \sqrt{2d}n -\frac{\log n}{2\sqrt{2d}}+\frac{ 4(\log \log n)}{\sqrt{2d}}$ we obtain that
\begin{multline}
\bbP \left[ Y_{n,{\bf i}}+Z_{n,{\bf i}} - A_n\ge 0 \ ; \cB_q \right] \\
\le \frac{C_q(\log n)}{\sqrt{n}}\int  (q+\log n +\sqrt{2d}n-u)_+  e^{- \frac{u^2}{2n}-c(A_n-u)^2_+} \dd u
 \le \frac{C'_q  e^{- dn}}{n(\log n)^{2}}.
\end{multline}
Using Borel-Cantelli after summing over ${\bf i}\in I_n$ and $n\ge 0$, we obtain that for each $q\in \bbN$ 
\begin{equation}
 \left( \sup_{n\ge 0} \sup_{i \in I_n}  Y_{n,{\bf i}}+Z_{n,{\bf i}} - \sqrt{2d}n + \frac{\log n}{2\sqrt{2d}}-\frac{ 4(\log \log n)}{\sqrt{2d}} \right)\ind_{\cB_q}<\infty.
\end{equation}
Letting $q\to \infty$ we conclude the proof. \qed

\begin{proof}[Proof of Lemma \ref{ggttail}]
 The Borrel-TIS inequality (see \cite[Theorem 2.1]{Adler}) implies that the fluctuation of the maximum of a Gaussian field  around its mean are sub-Gaussian, that is 
 \begin{equation}\label{BTIS}
  \bbP[Z_{n,{\bf i}}\ge \bbE\left[ Z_{n,{\bf i}} \right]+u ] \le 2e^{-\frac{u^2}{2\sigma_{n}}}
 \end{equation}
where 
$\sigma_{n}:= \max_{(t,x)\in  A_{n,{\bf i}}} \bbE[ (\bar X_t(x)-Y_{n,{\bf,i}})^2]= 1+ \max_{x\in [0,e^{-n}]^2}\bbE[ (\bar X_n(x)-\bar X_n(0))^2].$\\
The maximized quantity  is simply equal $2(n-\bar K_n(0,x))$, and thus  using \eqref{bb5}, we have $\sup_{n\ge 0} \sigma_n<\infty$.
To deduce \eqref{kilaro} from  \eqref{BTIS}, we want a uniform  bound on $\bbE\left[ Z_{n,{\bf i}} \right]$ that is
\begin{equation}\label{labelouze}
\sup_{n\ge 0} \sup_{i \in I_n} \bbE\left[ Z_{n,{\bf i}} \right]=\sup_{n\ge 0} \bbE\left[ \max_{(t,x)\in[0,1]^{d+1}} \bar X_{n+t}(e^{-n}x) \right]<\infty
\end{equation}
(the equality comes from translation invariance and the fact that $Y_{n,{\bf i}}$ is centered). To achieve this, we need to control, uniformly in $n$, the modulus of continuity of the canonical metric associated with the field 
$(\bar X_{n+t}(e^{-t}x))_{(t,x)\in[0,1]^{1+d}}$, which is defined by
$$d_n((s,x),(t,y)):= \bbE\left[ (\bar X_{n+s}(e^n x)-\bar X_{n+t}(e^n y))^2\right]^{1/2}.$$
It is sufficient to prove that 
$d_n$ is uniformly H\"older continuous.
Given $t\ge s$ we have
\begin{multline}\label{bb5}
\bbE\left[ (\bar X_s(x)-\bar X_t(y))^2\right]= (t-s)+2(s- \bar K_s(x,y)) \le (t-s)+ 2\int^s_0 (1-Q_u(x,y)) \dd u
\\ \le (t-s)+ C\int^s_0 e^{2|u|}|x-y|^2 \dd u \le (t-s)+C e^{2|s|}|x-y|^2.
\end{multline}
Hence for any  $(s,x),(t,y) \in [0,1]^{d+1}$ we have (for a different constant $C$), 
\begin{equation}\label{bdt}
d_n((s,x),(t,y))^2 \le |t-s|+ C |x-y|^2.
\end{equation}
Using either Fernique's majorizing measure techniques 
(applying \cite[Theorem 4.1]{Adler} with Lebesgue measure on $[0,1]^{1+d}$) or Garsia-Rodemich-Rumsey inequality (see for instance \cite[Corollary 4]{LectureZ}) we obtain that \eqref{bdt} implies \eqref{labelouze} and conclude the proof.
\end{proof}

 \subsection{Proof of Propositions \ref{lateknik}}\label{tekos2}

Let us first prove \eqref{tek1}.
By Cameron-Martin formula, the $e^{\sqrt{2d}(X_t(x)+X_t(y))}$ term shifts the mean of $\bar X_s(x)$ and $\bar X_s(y)$ 
 by  $\sqrt{2d}(s+\bar K_s(x,y))$ for $s\in[0,t]$. As a result we obtain that 
 \begin{equation}\label{a11}
 \!\! \bbE\left[W^{(q)}_t(x) W^{(q)}_t(y)\right] =
  e^{2d K_t(x,y)} \bbP\left[ \forall s\in [0,t],  \bar X_s(x) \! \vee  \! \bar X_s(y)\le  q - \sqrt{2d} \bar K_s(x,y) \right].
 \end{equation}
To estimate the probability on the right-hand side, we perform a couple of simplifications. 
Recalling the definition of $u(x,y,t)$, from \eqref{thelogbound}, there exists  $C>0$ such that 
\begin{equation}
\label{krootz}
\bar K_s(x,y) \ge s\wedge u - C.
\end{equation}
For the remainder of the proof, we set $q'=q+C\sqrt{2d}$. Another observation is that $\bar X_{\cdot}(x)$ and $\bar X_{\cdot}(y)$ play symmetric roles so that 
adding the restriction $\bar X_u(x)\le \bar X_u(y)$ only changes the probability of the event  in the r.h.s.\ of \eqref{a11} by a factor $1/2$. It is thus smaller than 
\begin{multline}\label{wopwopwop}
2 \bbP\left[ \forall s\in [0,t],  \bar X_s(x) \! \vee  \! \bar X_s(y)\le  q' - \sqrt{2d}(u\wedge s) \ ; \ \bar X_u(x)\le \bar X_u(y) \right] 
 \\ \le 2 \bbP\left[ F_1 \cap F_2 \cap F_3\right]= 2 \bbE\left[ \ind_{F_1} \bbP[ F_2\cap F_3  \ | \ \cF_u] \right].
 \end{multline}
 where the events $(F_i)^3_{i=1}$ are defined by
 \begin{equation}
  \begin{split}
   F_1&:=  \{ \forall s\in [0,u],  \bar X_s(x)\le  q' - \sqrt{2d} s \},\\
   F_2&:=\{ \forall s\in [0,t-u],  X^{(u)}_s(x)\le  q'-\sqrt{2d}u - \bar X_u(x) \},\\
   F_3&:=\{ \forall s\in [0,t-u],  X^{(u)}_s(y)\le  q' -\sqrt{2d}u - \bar X_u(x) \}.
  \end{split}
 \end{equation}
 The $\bar X_u(x)$ appearing in the definition of $F_3$ is not a typo: the restriction $\bar X_u(x)\le \bar X_u(y)$   implies that $F_3$ (like $F_1$ and $F_2$) is included in the event
  in the l.h.s.\ of   \eqref{wopwopwop}.
From the observations of Section \ref{considerations}, $X^{(u)}_{\cdot}(x)$ and $X^{(u)}_{\cdot}(y)$ are independent and independent of $\mathcal F_u$.
Thus $F_2$ and $F_3$ are conditionally independent given $\cF_u$, and  using Lemma \ref{stupid} we have
\begin{equation}\label{cronde1}
  \bbP\left[ F_2\cap F_3 \ | \ \cF_u\right] \le 1\wedge \left(\frac{2(q'-\sqrt{2d}u -\bar X_u(x))^2_+}{\pi(t-u)}\right).
\end{equation}
To compute $\bbE\left[ \ind_{F_1} \bbP[ F_2\cap F_3  \ | \ \cF_u] \right]$, we use
 \eqref{bridjoux} and obtain that 
\begin{equation}\label{cronde2}
 \bbE[F_1 \ | \ \bar X_u(x)]\le 1\wedge \left(\frac{2q'(q' -\sqrt{2d}u-\bar X_u(x))_+}{u}\right).
\end{equation}
Putting together \eqref{cronde1} and \eqref{cronde2} and integrating w.r.t.\  $z=(q'-\sqrt{2d}u-\bar X_u(x))$ we have
\begin{multline}\label{joke}
\bbP\left[ F_1 \cap F_2 \cap F_3\right]\le  \int^{\infty}_0 \frac{e^{-\frac{(z-q'+\sqrt{2d} u)^2}{2u}}}{\sqrt{2\pi u}}
 \left(1\wedge \frac{2q'z}{u}  \right)     \left(1\wedge \frac{z^2}{\pi(t-u)}\right) \dd z\\
 \le \frac{C e^{-du}}{(u + 1)^{3/2}[(t-u)+ 1]}.
\end{multline}
 Since in \eqref{a11}  we have $e^{2d K_t(x,y)}\le C e^{2du}$ (cf. \eqref{thelogbound}) this concludes the proof of \eqref{tek1}.
To prove \eqref{tek2} we notice that from Cameron Martin formula we have
\begin{equation*}
  \bbE\left[\! W^{(q)}_\gep(x) W^{(q)}_\gep(y) \!\right] \! =\!
  e^{2d K_\gep(x,y)} \bbP\left[ \forall s\in [0,t],  \bar X_s(x) \! \vee  \! \bar X_s(y)\le  q - \sqrt{2d} \bar K_{s,\gep,0}(x,y) \right].
 \end{equation*}
Recalling the definition of $v(x,y,\gep)$, the bound \eqref{thelogbound} implies that  $K_\gep(x,y)\le v+C$
and 
$\bar K_{s,\gep,0}(x,y) \ge s\wedge v-C.$
We can thus repeat the proof of \eqref{tek1}  with $u$ replaced by $v$.

\medskip

\noindent To prove \eqref{tek3}-\eqref{tek4}, we set $\beta=\sqrt{2d}-\alpha$. Using \eqref{krootz} and Cameron-Martin formula
\begin{equation*}
  \bbE\left[ W^{\alpha,(q)}_t(x) W^{\alpha,(q)}_t(y) \right]  \le 
  e^{\alpha^2 K_t(x,y)} \bbP\left[ \forall s\in [0,t],  \bar X_s(x) \! \vee  \! \bar X_s(y)\le  q' +\beta s-\alpha (u\wedge s) \right].
 \end{equation*}
 We use \eqref{thelogbound} to bound  the exponential prefactor. We obtain (omitting $\alpha$ and $q$ in the r.h.s.\ for readability)
 \begin{equation}\begin{split}\label{dumb}
   \bbE\left[ W_t(x) W_t(y) \right]&\le  C  e^{\alpha^2 u } \bbP\left[ \forall s\in [0,u],  \bar X_s(x)  \le   q'- (2\alpha-\sqrt{2d}) s\right],\\  
  \lim_{t\to \infty} \bbE\left[ W_t(x) W_t(y) \right]&\le  C  e^{\alpha^2 w } \bbP\left[ \forall s\ge 0,  \  \bar X_s(x) \! \vee  \! \bar X_s(y)  \le   q' +\beta s-\alpha (w\wedge s)\right].
   \end{split}
 \end{equation}
We introduce the following events (for $r\ge 0$)
 \begin{equation}
  \begin{split}
   G^r_1&:=  \{ \forall s\in [0,r],\  \bar X_s(x)\le  q' - (2\alpha-\sqrt{2d}) s  \},\\
   G_{2}&:=\{ \forall s\ge 0, \ X^{(w)}_s(x)\le  q'-  (2\alpha-\sqrt{2d})w - \bar X_w(x)+\beta s \},\\
   G_{3}&:=\{ \forall s\ge 0, \ X^{(w)}_s(y)\le  q'-  (2\alpha-\sqrt{2d})w - \bar X_w(x)+\beta s \}.
  \end{split}
 \end{equation}
To prove \eqref{tek3}, we assume that $u\ge 1$ and $\alpha>\sqrt{2d/3}$ (if not then the bound $Ce^{\alpha^2 u^2}$ from  \eqref{dumb} is good enough). Integrating over $\bar X_u(x)$ and using Lemma \ref{stupid} we have 
\begin{equation}
 \bbP[G^u_1]\le \frac{1}{\sqrt{2\pi u}}\int_{\bbR} e^{-\frac{z^2}{2u}} \frac{2q'(q'- (2\alpha-\sqrt{2d}) u-z)_+}{u}\dd z \le C e^{-\frac{(2\alpha-\sqrt{2d})^2}{2}u} u^{-3/2}.
\end{equation}
We can conclude that \eqref{tek3} holds by observing that 
\begin{equation}\label{ckomca}
\alpha^2-\frac{(2\alpha-\sqrt{2d})^2}{2}=d- (\alpha-\sqrt{2d})^2\le d.
\end{equation}
 For \eqref{tek4} starting with the second line in\eqref{dumb}  we  proceed as in \eqref{wopwopwop} and  obtain  
 \begin{equation}
    \lim_{t\to \infty} \bbE\left[ W_t(x) W_t(y) \right]\le 2C e^{\alpha^2 w} \bbP\left[G^w_1\cap G_2\cap G_3 \right]
 \end{equation}
Repeating the steps leading to \eqref{joke} and replacing \eqref{cronde2} by (cf. Lemma \ref{stupid})
\begin{equation}
 \bbP\left[ G_2\cap G_3  \ | \ \cF_u \right]\le 4\beta^2(q' -(2\alpha-\sqrt{2d})w- \bar X_w(x))
\end{equation}
  we obtain (the variable of integration is $z=q' -(2\alpha-\sqrt{2d})w - \bar X_w(x)$)    
\begin{equation*}
\bbP\left[ F_1 \cap \bar G_2 \cap \bar G_3\right]\le  \int^{\infty}_0 \frac{e^{-\frac{(z-q'+(2\alpha-\sqrt{2d}) w)^2}{2w}}}{\sqrt{2\pi w}}
 \left(1\wedge \frac{2q'z}{w}  \right)    4\beta^2z^2    \dd z\\
 \le \frac{C \beta^2 e^{\frac{(2\alpha-\sqrt{2d})^2w}{2}}}{(w + 1)^{3/2}}.
\end{equation*}
Recalling \eqref{dumb} and \eqref{ckomca}, this is sufficient to conclude the proof of \eqref{tek4}.

\medskip

To prove \eqref{hiphip}, we set $\delta=|E|$. Without loss of generality we can assume that $|\delta|\le 1$.
Expanding the second moment, and  the change of variable $r=-\log 1/|x-y|$ when integrating over $y$ we obtain
\begin{multline}
 \bbE\left[ M^{(q)}_t(E)^2\right]\le \int_E\left(\int_{B(x,\delta)} W^{(q)}_t(x)W^{(q)}_t(y)\right) \dd x
 \\ \le C \gl(E)\int^{\infty}_{\log (1/\delta)}  \frac{ e^{-dr} e^{dr\wedge t} \dd r}{ (r\wedge t+1 )^{3/2} (t-r\wedge t+1) }.
\end{multline}
The factor $\gl(E)$ can be bounded by $\delta^d$. Assuming that $t\ge 2\log (1/\delta)$ we have
\begin{multline}\label{joke3}
 \int^{\infty}_{\log (1/\delta)}  \frac{ e^{-dr} e^{d(r\wedge t)} \dd r}{ (r\wedge t+1 )^{3/2} (t-r\wedge t+1) } \\
 \le  \frac{2}{t} \int^{t/2}_{\log (1/\delta)} \frac{\dd r }{(r+1)^{3/2}} + \frac{1}{(t/2+1)^{3/2}}\int^t_{t/2} \frac{\dd r}{t-r+1}
 +\frac{1}{(t+1)^{3/2}}\int^{\infty}_t e^{-d(r-t)} \dd r
 \\
 \le C \left( t^{-1} (\log 1/\delta)^{1/2} +  t^{-3/2}\log t\right).
\end{multline}
This proves \eqref{hiphip} and the two first lines of  \eqref{whendelta} can be obtained similarly.\\
For the third line of \eqref{whendelta}, using \eqref{tek3} (which of course valid if also $u$ is replaced by $w$), and applying Fatou to the positive integrable function $(Ce^{dw}(1+w)^{3/2}-\bbE\left[W_t(x)W_t(y) \right])$ we obtain that the $\limsup$ of in the l.h.s.\ of \eqref{adrenaline} is smaller than the integral of the $\limsup$ of the integrand. With \eqref{tek4} this yields (recall that $\beta=\sqrt{2d}-\alpha$)
\begin{equation}\label{adrenaline}
\limsup_{t\to \infty} \int_{E^2} \bbE\left[W^{\alpha,(q)}_t(x)W^{\alpha,(q)}_t(y)\right]\ind_{\{|x-y|\le \delta\}} \dd x \dd y\le C\beta^2 \int_{E^2} \frac{e^{dw}\ind_{\{|x-y|\le \delta\}} }{(w+1)^{3/2}}\dd x \dd  y,
\end{equation}
and it is easy to conclude from there.
\qed

\subsection{Estimates on the covariance}\label{prosteak}
Let us start with \eqref{steak}
Since from \eqref{defqt} we have $Q_u(x,y)\le \ind_{\{|x-y|\le e^{-u}\}}$ we obtain that 
for $|x-y|\le 1$,
\begin{equation}
 \bar K_t(x,y)\le \int^{t}_0  \ind_{\{|x-y|\le e^{-u}\}}\dd u = t\wedge \log_+\frac{1}{|x-y|}.
\end{equation}
Now for a lower bound, using the $C^2$ regularity of $\kappa$ around $0$, \eqref{defqt} implies that 
$$Q_u(x,y)\ge 1-Ce^{2u}|x-y|^2$$
This implies that if $|x-y|\le e^{-s}$ then
$$ \bar K_s(x,y)\ge  s- C\int^t_0 e^{2u}|x-y|^2 \ge s-C'$$
Using the monotonicity in $s$, applying this bound for $s=t\wedge \log_+( 1/|x-y|)$ we obtain
\begin{equation}
 \bar K_s(x,y)\ge t\wedge \log_+( 1/|x-y|)- C.
\end{equation}
The upper and lower bound remain valid with an additional constant when adding $K_0(x,y)$ which is continuous and hence locally bounded, 
and convoluting the upper or the lower bound with $\theta_\gep$ on one or both variables yields \eqref{thelogbound} for $K_{t,\gep}$ and $\bar K_{t,\gep,0}$.

\section{Weak convergence of measure}\label{weakconvmeas}

Let us first present a metric which corresponds to the topology of weak convergence on $\bbR^d$ (for locally finite, nonnegative measures). We let $\pi_R$ denote the Lévy-Prokhorov metric between two finite measures on $B(0,R)$
\begin{equation}
 \pi_R(\mu, \nu):= \inf \{ \gep>0 \ : \ \forall A\in \cB_R, \     \mu(A^{\gep}) \le \nu(A)+\gep \text{ and } \nu(A^{\gep}) \le \mu(A)+\gep   \}
\end{equation}
where $\cB_R$ is collection of Borel subsets of $B(0,R)$ and 
$$A^{\gep}:= \{ x\in B(0,R) \ : \  \exists y\in A, |y-x|\le \gep \}.$$
For locally finite measures $\mu$ and $\nu$ on $\bbR^d$, the metric $d$ defined by
$$d(\mu,\nu):= \sum^{\infty}_{R=1} 2^{-R} \max( \pi_R(\mu, \nu),1),$$
where in the summand $\mu$ and $\nu$ are identified with their restriction on $B(0,R)$, generates the topology of weak convergence of \eqref{wikiwik}.
We say that a sequence of random measure $M_n$ converges \textit{weakly in probability} if 
\begin{equation}
 \forall \gep\in(0,1], \ \lim_{n\to \infty}\bbP\left[ d(M_n,M)>\gep \right]=0.
\end{equation}
Since the topology of weak convergence is separable, convergence in probability  implies that there exists a subsequence that converges weakly almost surely. Recall that $\gl$ denote the Lebesgue measure.

\begin{proposition}\label{weako}
 Let $M_n$ be a sequence of non-negative random measures and let us assume that for any fixed $E\in \cB_b$, 
 $M_n(E)$ converges in probability towards a finite limit $\bar M(E)$.
 Then the following holds
 \begin{itemize}
  \item [(i)] For any $f\in B_b$, $M_n(f)$ converges in probability towards a finite limit $\bar M(f)$.
  \item [(ii)]
  There exists a random measure $M$  towards which $M_n$ converges in probability.
  \item [(iii)] If additionally there exist $K>0$ such that $\bbE[M_n(E)]\le K\gl(E)$ for all $E\in \mathcal B_b$ and $n$, then for all    $f\in B_b$, $\bbP[\bar M(f)= M(f)]=1$.
 \end{itemize}
If the convergence of  $M_n(E)$ holds a.s.\ then the convergence $(i)$ and $(ii)$ also hold a.s.

\end{proposition}

\begin{proof}
 For $(i)$ we observe that the statement is obviously true for simple functions with bounded support.
 Given $f\in B_b$, and $k\ge 1$, setting $E:=\{x \ : \ f(x)\ne 0\}$ we can find an increasing sequence of simple functions $(f_k)_{k\ge 1}$  such that 
 $0\le (f-f_k)\le k^{-1}\ind_E$. This entails
 \begin{equation}
0\le M_n(f)-M_n(f_k)\le k^{-1} M_n(E)
 \end{equation}
and hence that $M_n(f)$ converges to $\lim_{k\to \infty} \bar M(f_k)$. 
For $(ii)$ note that $(i)$ already entails the convergence of $(M_n(f))_{f\in F}$ for any finite family $F\subset C_c(\bbR^d)$. 
Hence it is sufficient to prove that the sequence $(M_n)_{n\ge 0}$ is tight for the weak convergence. For almost-sure convergence we must check that $\sup_n M_n(B(0,R))$ is finite for every $R\in \bbN$ 
and for convergence in probability we need to check that for every $R\in \bbN$ the sequence  $M_n(B(0,R))$ is tight. In both cases this is a consequence
the fact that, by assumption $M_n(B(0,R))$ converges.

\medskip

\noindent For $(iii)$, given  $g\in C_c(\bbR^d)$ nonnegative, since by $(ii)$, $M_n(g)\to M(g)$   Fatou  implies that
\begin{equation}\label{bdeux}
\bbE\left[ M(g) \right]\le  K \int g(x)\dd x.
\end{equation}
This implies that the measure $E\mapsto \bbE[M(E)]$ is absolutely continuous with respect to Lebesgue with density bounded above by $K$.
Now given  $f\in B_b$ and $\delta$ we let  $g\in C_c(\bbR^d)$ be such that $\int |f-g|\dd x \le \delta$ we have
\begin{equation}
\bbE\left[ |M(f)-\bar M(f)|\right]\le \bbE\left[ |M(f)-M(g)|\right]+\bbE\left[ |\bar M(f)-M(g)|\right].
\end{equation}
Our observation concerning the density of $\bbE[M(\cdot)$ implies that the first term is smaller than $K \delta$.
The second term is smaller than  $\bbE\left[ \bar M(|f-g|)\right]$ which by Fatou is smaller than
\begin{equation}
\liminf_{n\to \infty} \bbE\left[ M_n(|f-g|) \right]\le K\delta. 
\end{equation}
Since $\delta$ is arbitrary this is sufficient to conclude.
\end{proof}

\section{Extending the main result beyond star-scale invariance} \label{sobolev} 

Using a decomposition result from  \cite{junnila2019}, we show that our main theorems can be extended to sufficiently regular kernels defined on an arbitrary domain $\cD$.  
Let us recall the definition for the Sobolev space with index $s\in \bbR$ on $\bbR^k$ 
which is the Hilbert space of complex valued function associated with the norm 
\begin{equation}
 \| \varphi \|_{H^{s}(\bbR^k)}:= \left( \int_{\bbR^{k}} (1+|\xi|^2)^s |\hat \varphi(\xi)|^2 \dd \xi \right)^{1/2},
\end{equation}
where $\hat \varphi(\xi)$ is the Fourier transform of $\varphi$  defined for $\varphi\in C^{\infty}_c(\bbR^k)$ by
$ \hat \varphi(\xi)= \int_{\bbR^{k}} e^{i\xi x} \varphi(x)\dd x.$
For $U\subset \bbR^k$ open, the local Sobolev space $H^{s}_{\mathrm{loc}}(U)$ denotes the function which belongs to  $H^{s}(U)$ after multiplication by an arbitrary smooth function with compact support
\begin{equation}\label{locsob}
H^{s}_{\mathrm{loc}}(U):= \left\{  \varphi : U \to \bbR  \ | \ \rho\varphi\in
 H^{s}(\bbR^d) \text{ for all } \rho\in C^{\infty}_c(U)\right\},
\end{equation}
where above $\rho\varphi$ is identified with its extension by zero on $\bbR^k$.
The following result is a consequence of  \cite[Theorem 4.5]{junnila2019}.

\begin{proposition}\label{dapopo}
 If $K$ is a positive definite kernel on $\cD$ that can be written in the form \eqref{fourme} with 
 $L\in H^{s}_{\mathrm{loc}}(\cD^2)$ with $s>d$, then for every $z\in \mathcal D$, there exist $\delta_z>0$ such that the restriction of 
 $K$ to $B(z,\delta_z)$ has an almost star-scale invariant part.
\end{proposition}
\noindent A consequence of the above is the following generalization of Theorem \ref{mainresprim}. A similar generalization can be made for Theorem  \ref{ziklob}.
The set of continuous resp. measurable function on $\cD$ vanishing outside a compact is denoted by $C_c(\cD)$ resp. $B_b(\cD)$. 
If $K$ is a positive definite kernel on $\cD$ that can be written in the form \eqref{fourme} then 
we define a centered Gaussian field $X$ indexed by $C_c(\cD)$ as in Section \ref{molly} and given $\gep\in (0,1]$ 
set $X_{\gep}(x):=\langle X, \theta_{\gep}\rangle$ for $x\in \cD_{\gep}$ where 
$\cD_{\gep}:=\{ x\in \cD \ : \ \forall y\in  \bbR^d \setminus \cD, \  |x-y|\ge 2\gep \}$ (the definition ensures that $\theta_{\gep}$ can be identified with an element of $C_c(\cD)$).
Then we define for $f\in B_b(\cD)$

$$ M_{\gep}(f):= \int_{\cD_{\gep}}e^{\sqrt{2d}X_{\gep}(x)-d \bbE[X_{\gep}(x)^2]}  \dd x.$$ 

\begin{theorem}\label{sobolevteo}
  If $K$ is a positive definite  kernel on $\cD$ that can be written in the form \eqref{fourme} with 
 $L\in H^{s}_{\mathrm{loc}}(\cD^2)$ with $s>d$.
 Then there exists a random measure $M'$on $\cD$  with dense support such that for any choice of mollifier $\theta$, $\sqrt{\pi\log (1/\gep)/2}M_{\gep}$ 
 converges weakly in probability towards a limit $M'$. For every $f\in B_b(\cD)$ we have $ \lim_{\gep\to 0} M^{(\sqrt{2d})}_{\gep}(f)= M'(f)$ in probability
 and for any compact subset $D\subset \cD$, we have (recall \eqref{phifou})
 \begin{equation}\label{jobi}
  \sup_{E\subset D} M'(E)/\phi(|E|)<\infty.
 \end{equation}

\end{theorem}

\begin{proof}[Sketch of proof]
 By Proposition \ref{weako}, it is sufficient to prove the convergence in probability of 
 $M^{(\sqrt{2d})}_{\gep}(E)$ for every $E\in \cB_d(\cD)$. Since $\bar E$, the topological closure of $E$ is compact, 
 using Proposition \ref{dapopo}, we can cover it by a finite number of balls $(B(z_i,\gd_i)_{i\in I}$.
 We can associate to this cover a partition of unity $(\rho_i)_{i\in I}$ (the $\rho_i$ are continuous, vanish outside $B(z_i,\gd_i)$ and $\sum_{i\in I}\rho_i(x)= 1$  for $x\in \bar E$).
 Applying Theorem \ref{mainres} and Proposition \ref{weako}, we obtain that $M^{(\sqrt{2d})}_{\gep}(\rho_i \ind_E)$ converges in probability for every $i\in I$ and hence $M^{(\sqrt{2d})}_{\gep}(\ind_{\rho_i} E)$ 
 converges for every $i$, and thus $M^{(\sqrt{2d})}_{\gep}(E)$.
 We deduce \eqref{jobi} from \eqref{gageli} in a similar fashion.
\end{proof}

\bibliographystyle{plain}
\bibliography{bibliography.bib}

\end{document}